\documentclass[11pt,reqno]{amsart}

\usepackage{amscd,amssymb,amsmath,amsthm}
\usepackage[arrow,matrix]{xy}
\usepackage{graphicx}
\usepackage{color}
\usepackage{cite}
\usepackage{hyperref}
\usepackage{caption, subcaption}
\usepackage{tikz}
\usepackage{tikz-cd}
\usepackage{tikz-3dplot}

\usetikzlibrary{calc}
\usetikzlibrary{positioning}
\usetikzlibrary{decorations.markings}
\usetikzlibrary{arrows,snakes,backgrounds}

\tikzstyle{place}=[circle,draw=black,fill=blue!20,thick,
				inner sep=0pt,minimum size=6mm]				
\tikzstyle{dot}=[circle,draw=black,fill=black!20,thick,
				inner sep=0pt,minimum size=1mm]
\tikzstyle{rec}=[rectangle,draw=black,fill=black!20,thick,
				inner sep=0pt,minimum size=1mm]				
\tikzstyle{pre}=[<-,shorten <=0pt, >=stealth',semithick]
\tikzstyle{post}=[->,shorten <=0pt, >=stealth',semithick]

\newtheorem{thm}{Theorem}
\newtheorem{defn}{Definition}

\newtheorem{rem}{Remark}

\numberwithin{equation}{section} 

\DeclareMathOperator{\bl}{\mathrm{BL}}

\DeclareMathOperator{\cp}{\mathrm{cp}}

\newcommand{\rotsubset}{\mathbin{\rotatebox[origin=c]{90}{$\subset$}}}

\def\e{\varepsilon}
\def\a{\alpha}
\def\b{\beta}
\def\g{\gamma}
\def\G{\Gamma}

\def\D{\Delta}

\def \L {\Lambda}

\def\s{\sigma}

\def\O{\Omega}
\def\o{\omega}

\def\L{\Lambda}

\def\Z{\mathbb{Z}}
\def\R{\mathbb{R}}
\def\N{\mathbb N}

\def\FF{\mathcal{F}}

\def\M{\mathcal M}
\def\GG{\mathcal G}

\def\SS{\subset\subset}

\begin{document}
\title[Gradient Gibbs measures]{
Gradient Gibbs measures and fuzzy transformations on trees 
}

\author{C. K\"ulske, P. Schriever}

\address{C.\ Kuelske\\ Fakult\"at f\"ur Mathematik,
Ruhr-University of Bochum, Postfach 102148,\,
44721, Bochum,
Germany}
\email {Christof.Kuelske@ruhr-uni-bochum.de}

\address{P.\ Schriever\\ Fakult\"at f\"ur Mathematik,
Ruhr-University of Bochum, Postfach 102148,\,
44721, Bochum,
Germany}
\email {Philipp.Schriever-d8j@ruhr-uni-bochum.de}

\begin{abstract} We study Gibbsian models of unbounded integer-valued spins on trees which 
possess a symmetry under height-shift. We develop a theory relating boundary laws 
to gradient Gibbs measures, which applies also in cases 
where the corresponding Gibbs measures do not exist.  
Our results extend the classical theory of Zachary \cite{Z83} 
beyond the case of normalizable boundary laws, which implies existence of Gibbs measures, 
to periodic boundary laws.  
We provide a construction for classes of 
tree-automorphism invariant gradient Gibbs measures in terms of mixtures of pinned measures, 
whose marginals to infinite paths on the tree are random walks in a $q$-periodic environment. 
Here the mixture measure is 
the invariant measure of a finite state Markov chain 
which arises as a mod-$q$ fuzzy transform, and which 
governs the correlation decay. 

The construction applies for example to SOS-models and discrete Gaussian models and delivers 
a large number of gradient Gibbs measures. 
We also discuss relations of certain gradient Gibbs measures to Potts and Ising models. 
\end{abstract}

\maketitle

{\bf Mathematics Subject Classifications (2010).} 82B26 (primary);
60K35 (secondary)

{\bf{Key words.}} 
Gibbs measure, gradient Gibbs measure, pinned gradient measure, fuzzy transform, tree. 
\section{Introduction}

There is renewed interest in the study of spin models on trees. A tree will in our context be 
a graph with countable vertex set which has no loops. On each vertex there is attached 
a random variable (spin) with values in a given local state space $\Omega_0$. 
Tree models are intrinsically interesting, and show phase transition behavior 
with sometimes greater richness than lattice models. 
Branching Brownian motion and related processes indexed by trees appear in the study of spin-glasses, see \cite{Bo16} and references therein. 

Another interest comes indirectly 
from models which are indexed by random graphs.  
Often one considers random graphs  which may contain loops but 
which at least locally look like trees with large probability \cite{DMIsing}, \cite{DMgm}, \cite{Do14}, \cite{GGHP}, \cite{MoMS12}. If one wants to understand those models 
on graphs one needs a safe understanding of the tree model first. 

There are interesting and hard questions which are open about spin models on regular trees, even if the local spin space $\Omega_0$ is finite.  
An example of 
such a question is the understanding of so-called reconstruction transitions,
including the 
determination of the reconstruction thresholds in terms of model parameters, see
\cite{FK}, \cite{JM04},\cite{KS},\cite{KuRo14}, \cite{M}, \cite{PP2010}, \cite{Sly11}. 
The problem is equivalent to decide whether a given Gibbs state is extremal 
in the simplex of all Gibbs states. It is even more delicate to describe its extremal decomposition \cite{Sh16}.
 See also \cite{BiEnEn16} for a recent work 
on the Ising model in a non-homogeneous field.

\subsection{Known theory: boundary laws, tree-indexed Markov chains for finite-state space models}
We restrict ourselves in the present paper mostly to regular trees and consider mainly 
infinite-volume states which possess all symmetries of the underlying tree graph. 
There is a very clear and complete presentation of the theory for finite spin space 
models which is presented in the textbook by Georgii, Chapter 12 \cite{Ge88}. A key notion here 
is that of a {\em boundary law} or {\em entrance law} 
which is a non-normalized distribution 
on the local spin space satisfying a certain non-linear fixed point equation which depends on the interaction 
potential $\Phi$ as a parameter and is 
obtained via a tree recursion 
\eqref{eq:bl}. The theory asserts that those Gibbs measures which are also {\em tree-indexed 
 Markov chains} (or, equivalently {\em splitting Gibbs measures}, for definitions see below)
 are in one-to-one correspondence with boundary laws, and it gives 
 the finite-volume marginals in terms of the boundary laws. 
Given the interaction potential $\Phi$ of the model, a boundary law $l$ and 
 the transition matrix $P$ of the Markov chain can be computed from each other.  
 This we symbolically depict as  
  $$P\leftrightarrow^\Phi l.$$
Also there is a theorem saying that extremal Gibbs measures are always Markov chains, while there may 
be Markov chain Gibbs measures which are non-extremal in the set of all Gibbs measures. 
This theory extends from finite to countable state space models, but it does 
 so only  {\em under the further assumption} of normalizability of the boundary law 
(see \cite{Z83}, formula 3.5). We also refer to the examples in Chapter 8 
of \cite{Ro}, which are constructed under this normalization assumption. 

How can we think of such states?
In this regime the marginals to infinite paths on the tree of the Gibbs measure are 
random walks which are localized in the height-direction, and have an invariant probability measure which is computable 
in terms of the boundary law. 

What happens to {\em non-normalizable boundary law 
solutions}Ê was not systematically investigated.

\subsection{Known theory: Gradient models on lattices}
Models with non-compact spin-space (with main examples $\R,\Z$)
which are invariant under a joint height-shift of all values of the 
spin-variables are well-known 
in statistical mechanics, under the names {\em interface models} or {\em gradient models}, see e.g.
\cite{BK07}, \cite{BK94}, \cite{BK96}, \cite{BDZ}, \cite{CK12}, \cite{CK15}, \cite{DGI}, \cite{Fu05}, \cite{FS97}, \cite{vEK08}.
In the important case that there is only a nearest neighbor pair interaction, this is then given via a potential function $U$ 
acting on differences of the spins at neighboring sites.  

Due to the non-compactness of the local spin-space, existence of Gibbs measures 
in the infinite volume is not to be taken for granted, and it may happen or may not happen, depending on the nature 
of the graph, in the lattice case its dimension, and parameters (inverse temperature, coupling strength, pinning 
forces) of the model. 
A proper infinite-volume Gibbs measure exists iff 
the interface in infinite volume is stable.  An important example where a Gibbs measure 
in the infinite volume does not exist
is provided by the lattice version of the Gaussian free field 
in dimension $d=2$, see \cite{FrPf81}, and its distribution-valued relative 
in continuous space \cite{S07}. 
The common way out to have a translation-invariant infinite-volume measure 
also in important cases when the infinite-volume {\em Gibbs} measure does not exist,  
is to divide out the translational degree of freedom in the height-direction and 
consider the {\em gradient Gibbs} measures (GGMs). 
A simple example is the pinning of a one-dimensional nearest neighbor ordinary random walk, 
conditioned to height zero in the origin. 
The increments are independent, and the measure on the increments 
is shift-invariant.


For lattice models there is a theory proving existence and uniqueness 
of GGMs with a fixed tilt (which applies in particular for flat horizontal interfaces) 
under the assumption of uniformly strictly convex potentials in dimensions $d=2$ developed by Funaki and Spohn
\cite{FS97}. (See however Remark 4.4. of \cite{Fu05} on existence for non-convex potentials.) 
This has an extension to random models 
\cite{CK12}, \cite{CK15} in dimensions $d\geq 3$, while for $d=2$ such random gradient states cannot exist \cite{vEK08} since they are locally destabilized by the influence of quenched randomness.

\subsection{Our question: How can gradient Gibbs measures on trees be constructed via boundary laws?}Ê

Are there states in the non-localized regime and can they be  constructed 
via boundary laws? 
It is the purpose of this paper to bring together the notion of {\em boundary law} 
and GGM for a tree model, and provide expressions for finite-volume 
marginals in terms of the boundary law. 
This generalizes the theory of Zachary \cite{Z83}, \cite{Z85}. 
Assume the local state space to be the integers, suppose we are in particular given 
a {\em height-periodic boundary law} with period $q$.  
The question appears naturally: 

Does such a non-normalizable boundary law (to which Zachary's theory can not apply) 
also correspond to a suitably constructed infinite-volume state?

\subsection{Outline of our Results} 
 
The answer has to be given in several steps. The short answer is yes, but the state has 
to be seen as a gradient state. 
The route to the construction and to the proof is to apply a useful mod-$q$ fuzzy transform, construct pinned measures, 
and average appropriately over the fuzzy chain. These steps will be described in more detail below. 
However, to give first an intuitive idea about the gradient states which will appear from this construction, let us look at their projections 
to an infinite path on the tree. These are (mixtures) of the increments of random walk paths. The increments are 
in general not independent, but they are given by a transition matrix which depends on 
the initial height in a $q$-periodic way. This transition matrix  depends on the boundary law. 
So, for a one-dimensional restriction on the tree,
 the paths look like random walks in a height-periodic environment. 
 The mixture measure has to be chosen very specifically 
 to recover translation-invariance {\em and } the gradient Gibbs property.  
 Clearly, for such measures a Gibbs measure can not exist, since the absolute heights of the 
 walks have no invariant probability distribution, only their increments. 

The paper is organized as follows. In Chapter 2 we give the general set-up of our model and the basic definitions. 
GGMs are obtained as the measures which satisfy the DLR equation w.r.t. the so-called {\em gradient specification}. In the case of lattice models the gradient specification $(\gamma')_{\L \subset\subset V}$ is usually simply defined as the restriction of the regular local Gibbs specification $(\gamma_\L)_{\L \subset\subset V}$ to the {\em naive outer 
sigma-algebra} which is generated by the height-shift invariant events that only depend on the increments (or the gradient) outside of $\L$.  
If we want to define gradient specifications in an analogous way, 
there is one peculiarity for tree graphs which must be treated properly for the theory to work: 
For trees the naive outer sigma-algebra must be replaced 
by  a strictly finer outer sigma-algebra which also 
retains the relative height information on the boundary of $\L$ together 
with the gradient configuration outside of $\L$. 
Indeed, since trees do not possess any loops the height differences between vertices outside finite sets $\L$ can in general not be recovered by the increments outside $\L$ alone (as it 
is the case for lattices) and the two sigma-algebras are different. 

In Chapter 3 we define measures on the space of gradient configurations via $q$-periodic boundary laws by pinning a class label $s \in \Z_q$ at some vertex. The existence of these measures is proven by showing applicability of Kolmogorov's extension theorem, where it turns out that the definition of a boundary law is tailor-made to guarantee the necessary consistency condition. We then show that these measures, which we call \textit{pinned gradient measures}, have a useful representation using transition matrices with 
an {\em additional  internal parameter}.   
This representation is analogous to but much more general than 
in the case of a finite state space, due to the added internal degree of freedom 
provided by the layer variable. 

In Chapter 4 we give a construction of 
gradient Gibbs measures by mixing the pinned gradient measures over the mod-$q$ fuzzy classes 
to recover both, the full gradient property and tree-homogeneity. 
Theorem \ref{MainThm} describes how this is done. 
The mixing measure $\a$ (appearing in the outmost integral) 
is the invariant measure of the $q$-state Markov chain 
which is naturally associated via a mod-$q$ fuzzy transform of the model. 

In Chapter 5 we give examples of GGMs that are constructed via $q$-periodic boundary laws.  
This involves a discussion on the relation between 
gradient models and the Potts and Ising model,  
exhibiting rich classes of GGMs 
with phase transitions. 
Furthermore it is shown how the associated $q$-state fuzzy Markov chain governs the correlation decay. 

%


\subsection{Comments and relation to work on preservation or loss of the Gibbs property under transformations}

Let us add some words comparing our present work (where we use the specific mod-$q$ fuzzy transformation to fuzzy spins
as a tool to {\em construct} a GGM on gradient variables) to related but different work where the behavior of Gibbs measures 
under different fuzzy transformations (or local coarse-grainings) was investigated.  
The integer-valued SOS model (on the lattice) was investigated by van Enter and Shlosman in \cite{vES98} under the transformation
which mapped a local spin to the sign field. In this situation non-Gibbsian measures were found, but with measure 
zero discontinuity points. 
The fuzzy Potts model on the tree 
was investigated in \cite{HK04} and again non-Gibbsian measures 
were proved to occur when the starting Gibbs measure was not the free Gibbs measure. 
In \cite{KuRo16} on the other hand, a fuzzy transformation to an Ising model which was 
adapted to the structure of the Gibbs measures was proved to be a Gibbs measure again. 

Summarizing, our present result has to be seen as a variation on the theme {\em Gibbs goes to Gibbs}. 
As important structural novelty note that the coarse-grained variables don't appear as a direct image 
of the gradient variables, but they are related via a coupling measure (see Appendix 6.1).  
The construction of a tree-invariant {\em gradient Gibbs measure} by mixing
pinned measures relies heavily on their relation given in terms of boundary laws.  
To appreciate this better 
we invite the reader also to consult Appendix 6.2 for a one-dimensional 
example of a {\em non-Gibbsian} gradient measure appearing 
by mixing of pinned measures not related via boundary laws.  

\section{The set-up and definitions}

Let $T=(V,E)$ be a locally finite connected tree with vertex set $V$ and edge set $E$. An unoriented bond $b\in E$ between two vertices $x,y \in V$ is denoted by $b = \{ x,y \}$. For the oriented edge going from $x$ to $y$ we write $\langle x,y \rangle$ and we call the set of all oriented edges $\vec E$. For a subset $A \subset V$ let $E(A)$ denote all the unoriented bonds connecting vertices in $A$, i.e. $E(A) = \{ \{x,y \} \in E \mid x,y \in A \}.$
Two vertices $x,y$ are called nearest neighbors, which we denote by $x \sim y$, if there exists an edge $b= \{ x,y \} \in E$. As $T$ has no loops there is a natural graph distance:  
For all vertices $x,y \in V$, there exists a unique self-avoiding path 
$$x = x_0 \sim x_1 \sim ...\sim x_n = y$$ 
in $V$ such that $\{ x_{k-1}, x_k \} \in E$ for all $1 \leq k \leq n$ and $x_k \neq x_j$ for all $k,j \in \{ 0, 1, ...,n \}$ with $ k \neq j$. Let $d(x,y)$ be the number of bonds of this unique self-avoiding path from $x$ to $y$, i.e. $d(x,y) = n$. 
 If $\L$ is a finite subset of vertices we write $\Lambda \subset\subset V$ and define its outer boundary as
 $$\partial \Lambda := \{ x \notin \Lambda : d(x,y) = 1 \mbox{ for some } y \in \Lambda\}.$$ 

Let a random field $(\phi_x)_{x \in V}$ of integer-valued random spin variables on the measurable space $(\Omega, \mathcal{F}) = (\Z^V, \mathcal{P}(\Z)^V)$ be given in its canonical form, i.e. $\phi_x : \Z^V \to \Z$ is defined by $\phi_x(\omega) = \omega(x) =\omega_x$, the projection onto the coordinate $x\in V$. Unlike in the Ising or Potts model, for instance, the state space of this random field is unbounded. A configuration $\omega \in \Z^V$ can be interpreted as a random realization of heights, labeled by the vertices of the tree graph (see Figure \ref{fig:lift12}). 
For any sub-volume $\L \subset V$ we let $\phi_\L: \Omega \to \Z^\L$ denote the projection onto the coordinates in $\L$ and define 
$\mathcal{F}_\L =\s(\{ \phi_y \mid y\in \L \})= \mathcal{P}(\Z)^\L$ to be the sigma-algebra 
which is generated by the height variables with sites in $\L$. 

\begin{figure}
\centering
	\begin{subfigure}{1\linewidth}
	\centering
	\beginpgfgraphicnamed{liftb}
	
	\tdplotsetmaincoords{56}{12}
		\begin{tikzpicture}[scale=1.4, tdplot_main_coords]
		
		\node[dot] (x) at (-1,0,0) {$$};
		\node at +([shift={(0.3,0,0.2)}]x) {$v_2$};
		
		\node[dot] (y) at (1,0,0)  {$$}
			edge[very thick]  	(x);
		\node (v_3) at  +([shift={(-0.3,0,0.2)}]y) {$v_3$};
		
		\node[dot] (w) at +([shift={(60:2)}]y) {$$}
			edge  		(y);
		\node[dot] (w1) at +([shift={(90:1)}]w) {$$}
			edge  		(w);
		\node[dot] (w2) at +([shift={(30:1)}]w) {$$}
			edge  		(w);	
		
		\node[dot] (b) at +([shift={(300:2)}]y) {$$}
			edge  	[very thick]	(y);
		\node at +([shift={(0.2,0,-0.4)}]b) {$v_4$};
		\node[dot] (b1) at +([shift={(270:1)}]b) {$$}
			edge  		(b);
		\node[dot] (b2) at +([shift={(330:1)}]b) {$$}
			edge  	[very thick]	(b);	
		\node at +([shift={(0.3,0,0)}]b2) {$v_5$};
		
		\node[dot] (a) at +([shift={(120:2)}]x) {$$}
			edge  	[very thick]	(x);
		\node at +([shift={(0.29,0,0.12)}]a) {$v_1$};
		\node[dot] (a1) at +([shift={(90:1)}]a) {$$}
			edge  		(a);
		\node[dot] (a2) at +([shift={(150:1)}]a) {$$}
			edge  [very thick]		(a);	
		\node at +([shift={(-0.3,0,0)}]a2) {$v_0$};
			
		\node[dot] (u) at +([shift={(240:2)}]x) {$$}
			edge  		(x);
		\node[dot] (u1) at +([shift={(210:1)}]u) {$$}
			edge  		(u);
			
		\node[dot] (u2) at +([shift={(270:1)}]u) {$$}
			edge  		(u);	
		
		\draw[thin, dashed, gray] (a2) -- +(0,0,3);
		\draw[thin, dashed, gray] (a2) -- +(0,0,-3);

		\draw[thin, dashed, gray] (a) -- +(0,0,3);
		\draw[thin, dashed, gray] (a) -- +(0,0,-3);
		
		\draw[thin, dashed, gray] (x) -- +(0,0,3);
		\draw[thin, dashed, gray] (x) -- +(0,0,-3);
		
		\draw[thin, dashed, gray] (y) -- +(0,0,3);
		\draw[thin, dashed, gray] (y) -- +(0,0,-3);
		
		\draw[thin, dashed, gray] (b) -- +(0,0,3);
		\draw[thin, dashed, gray] (b) -- +(0,0,-3);
		
		\draw[thin, dashed, gray] (b2) -- +(0,0,3);
		\draw[thin, dashed, gray] (b2) -- +(0,0,-3);
		
		\node[rec] (e1) at +([shift={(0,0,1)}]a2) {$$};
		\node[rec] (e2) at +([shift={(0,0,-2)}]a) {$$};
		\node[rec] (e3) at +([shift={(0,0,2)}]x) {$$};
		\node[rec] (e4) at +([shift={(0,0,-1)}]y) {$$};
		\node[rec] (e5) at +([shift={(0,0,2)}]b) {$$};
		\node[rec] (e6) at +([shift={(0,0,-2)}]b2) {$$};
		
		\foreach \from/\to in {e1/e2, e2/e3, e3/e4, e4/e5, e5/e6}
		\draw[dotted, thick] (\from) -- (\to);
		
		\node at +([shift={(0,0,-3.3)}]x) {$\Z$};
		
		\end{tikzpicture}
	\endpgfgraphicnamed
	\end{subfigure}
	\begin{subfigure}{1\linewidth}
	\centering%
	\beginpgfgraphicnamed{lifta}
	
	\tdplotsetmaincoords{90}{12}
		\begin{tikzpicture}[scale=1.4, tdplot_main_coords]
		
		\node[dot] (x) at (-1,0,0) {$$};
		\node (alx) at +([shift={(0.3,0,-0.3)}]x) {$v_2$};
		
		\node[dot] (y) at (1,0,0)  {$$}
			edge[very thick]  	(x);
		\node (aly) at +([shift={(-0.2,0,-0.3)}]y) {$v_3$};
		
		\node (w) at +([shift={(60:2)}]y) {$$}
			edge  		(y);
		\node (w1) at +([shift={(90:1)}]w) {$$}
			edge  		(w);
		\node[dot] (w2) at +([shift={(30:1)}]w) {$$}
			edge  		(w);	
		
		\node[dot] (b) at +([shift={(300:2)}]y) {$$}
			edge  	[very thick]	(y);
		\node (alb) at +([shift={(0.2,0,-0.3)}]b) {$v_4$};
		\node (b1) at +([shift={(270:1)}]b) {$$}
			edge  		(b);
		\node[dot] (b2) at +([shift={(330:1)}]b) {$$}
			edge  	[very thick]	(b);	
		\node at +([shift={(0.3,0,-0.3)}]b2) {$v_5$};
		
		\node[dot] (a) at +([shift={(120:2)}]x) {$$}
			edge  	[very thick]	(x);
		\node (al) at +([shift={(-0.2,0,-0.3)}]a) {$v_1$};
		
		\node (a1) at +([shift={(90:1)}]a) {$$}
			edge  		(a);
		\node[dot] (a2) at +([shift={(150:1)}]a) {$$}
			edge  [very thick]		(a);	 
		\node at +([shift={(-0.3,0,-0.3)}]a2) {$v_0$};
			
		\node (u) at +([shift={(240:2)}]x) {$$}
			edge  		(x);
		\node[dot] (u1) at +([shift={(210:1)}]u) {$$}
			edge  		(u);
			
		\node (u2) at +([shift={(270:1)}]u) {$$}
			edge  		(u);	
		
		\draw[thin, dashed, gray] (a2) -- +(0,0,3);
		\draw[thin, dashed, gray] (a2) -- +(0,0,-3);

		\draw[thin, dashed, gray] (a) -- +(0,0,3);
		\draw[thin, dashed, gray] (a) -- +(0,0,-3);
		
		\draw[thin, dashed, gray] (x) -- +(0,0,3);
		\draw[thin, dashed, gray] (x) -- +(0,0,-3);
		
		\draw[thin, dashed, gray] (y) -- +(0,0,3);
		\draw[thin, dashed, gray] (y) -- +(0,0,-3);
		
		\draw[thin, dashed, gray] (b) -- +(0,0,3);
		\draw[thin, dashed, gray] (b) -- +(0,0,-3);
		
		\draw[thin, dashed, gray] (b2) -- +(0,0,3);
		\draw[thin, dashed, gray] (b2) -- +(0,0,-3);
		
		\node[rec] (e1) at +([shift={(0,0,1)}]a2) {$$};
		\node[rec] (e2) at +([shift={(0,0,-2)}]a) {$$};
		\node[rec] (e3) at +([shift={(0,0,2)}]x) {$$};
		\node[rec] (e4) at +([shift={(0,0,-1)}]y) {$$};
		\node[rec] (e5) at +([shift={(0,0,2)}]b) {$$};
		\node[rec] (e6) at +([shift={(0,0,-2)}]b2) {$$};
		
		\foreach \from/\to in {e1/e2, e2/e3, e3/e4, e4/e5, e5/e6}
		\draw[dotted, thick] (\from) -- (\to);
		
		\node at +([shift={(0,0,-3.3)}]x) {$\Z$};
		
		\end{tikzpicture}
	\endpgfgraphicnamed

	\end{subfigure}
	
	\caption{An example of a random height field over the vertices along the path $\{ v_0, v_1,...,v_5\} \subset V$.}
	\label{fig:lift12}
\end{figure}
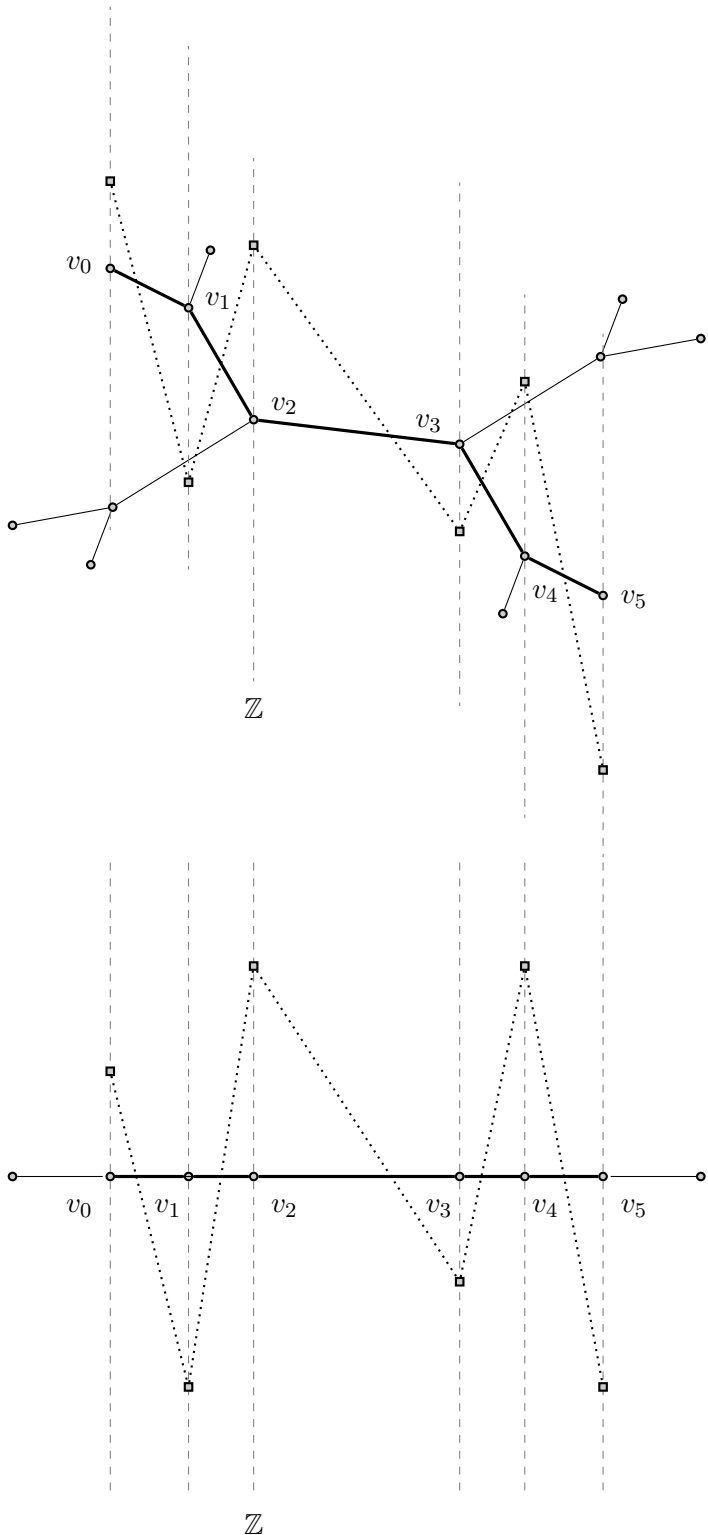

For $\omega = (\omega(x))_{x \in V}$ and $b = \langle v,w \rangle \in \vec E$ the \textit{height difference} along the edge $b$ is given by $\nabla \omega(b) = \omega_w - \omega_v$ and we also call $\nabla \omega$ the \textit{gradient field} of $\omega$. The gradient spin variables are now defined by $\eta_{\langle x,y \rangle} = \sigma_y - \sigma_x$ for each $\langle x,y \rangle \in \vec E$, and we define the projection mappings similarly as before for the height variables. Let us denote the state space of the gradient configurations by $\O^\nabla=\Z^{V}/ \Z = \Z^{\vec E}$ which becomes a measurable space with the sigma-algebra $\mathcal{F}^\nabla = \s(\{ \eta_b \mid b\in \vec E \}) = \mathcal{P}(\Z)^{\vec E}$. This is the space of all the possible gradient fields that can be prescribed by some height configuration $\omega \in \Z^V$, and trivially every gradient field $\zeta \in \Omega^\nabla$ gives a height configuration $\omega^{\zeta, \omega_x}$ for a fixed value of $\omega_x, x \in V$ by 
\begin{equation}
\omega^{\zeta, \omega_x}_y = \omega_x + \sum_{b \in \Gamma(x, y)} \zeta_b,  
\end{equation} 
where $\Gamma(x,y)$ is the unique self-avoiding path from $x$ to $y$. We note that due to the absence of loops, there is no plaquette condition, which is known for lattices to hold. 

Let some symmetric nearest-neighbor \textit{gradient interaction potential} $U_{b}: \Z \to \R$ be given for every $b=\{x,y\} \in E$, i.e. 
$$U_{b}(m) = U_{b}(-m)$$ for all $m\in \Z$.

To shorten our notation we sometimes write $\omega_b = \{ \omega_x, \omega_y \} \in \Z^2$ for edges $b = \{ x,y \} \in E$. Note that for all finite $\L \subset\subset V$ and any $\omega \in \Omega$, the quantity
\begin{equation}\begin{split}
H_\Lambda^U(\omega)  =  \sum_{b \cap \Lambda \neq \emptyset} U_b(\nabla \omega_b) 
\end{split}\end{equation}
exists and is finite. 
$H_\Lambda^U$ is called the \textit{Hamiltonian} in the finite volume $\L$ for $U$. 

\begin{defn}
The {\em local Gibbsian specification} corresponding to the Hamiltonian $H^U$ is defined as the family of probability kernels $(\gamma_\L)_{\L\subset\subset V}$ from $(\Omega, \mathcal{F}_{\L^c})$ to $(\Omega, \mathcal{F})$ by
\begin{equation}\begin{split}\label{spec}
\gamma_\Lambda (A\mid \tilde \omega) &= Z_\Lambda^{-1}(\tilde \omega) \int_A \exp\left( - \sum_{b\subset\L}U_b(\nabla \omega_b)  
- \sum_{i \in \L, j \in \L^c : i \sim j} U_{\{i,j\}} (\omega_i - \tilde \omega_j )\right) d\omega_\L, \\
\end{split}\end{equation}
where $Z_\L(\tilde \omega)$ denotes a normalization constant (or {\em partition function}) that turns the last expression into a probability measure for $\tilde \omega \in \Omega$.
\end{defn}
We define a family of functions $( Q_b )_{b\in E}$ with $Q_b : \mathbb{Z}_0 \to (0, \infty)$ s.t.
\begin{equation}\label{Qpotential}
Q_b (m)= \exp \left(- U_b(m)\right)
\end{equation}
for all $m\in \Z$. This family plays the role of {\em transfer operators}. 
Hence the Gibbsian specification admits the representation  
\begin{equation}\label{eq:specQ}
\gamma_\Lambda (A\mid \tilde \omega)  = Z_\Lambda^{-1}(\tilde \omega) \int_A  \prod_{b\subset\L} Q_b(\omega_b)  \prod_{x \in \L, y \in \L^c : x \sim y} Q_{xy} (\omega_x - \tilde \omega_y)  \; d\omega_\L.
\end{equation}
To have the partition functions finite, we assume throughout this paper that $Q_b \in l^1( \Z)$ for every $b \in E$. 

The reader may think of the concrete examples of the form $U(m)=\beta |m|^\a$ with $\a$ and $\b$ being positive constants, where the most popular cases 
are the SOS model obtained for $\a=1$ \cite{RS06}, and the so-called 
discrete Gaussian obtained for $\a=2$ \cite{S07}.

Note that the Hamiltonian $H$ changes only by a configuration-independent constant under the joint height-shift $\phi_x (\omega) \to \phi_x(\omega) + c$ of all spin variables $\phi_x(\omega), x \in V$ for the same constant $c\in\R$, which holds true for any fixed configuration $\omega \in \Z^V$. Using this invariance under the height shifts we can lift the probability kernels $\gamma_\L$ to kernels $\gamma'_\L$ on gradient configurations, as we will explain later on. 
First a warning is in order: On tree graphs the gradient specifications differ from those on lattices in more than one dimension since the complement of any finite set $\L \subset\subset V$ is disconnected (see Figure \ref{fig:grid}). Therefore the knowledge of a gradient configuration along the edges connecting vertices outside of $\L$ is not sufficient to reconstruct a boundary condition modulo overall height-shift. 

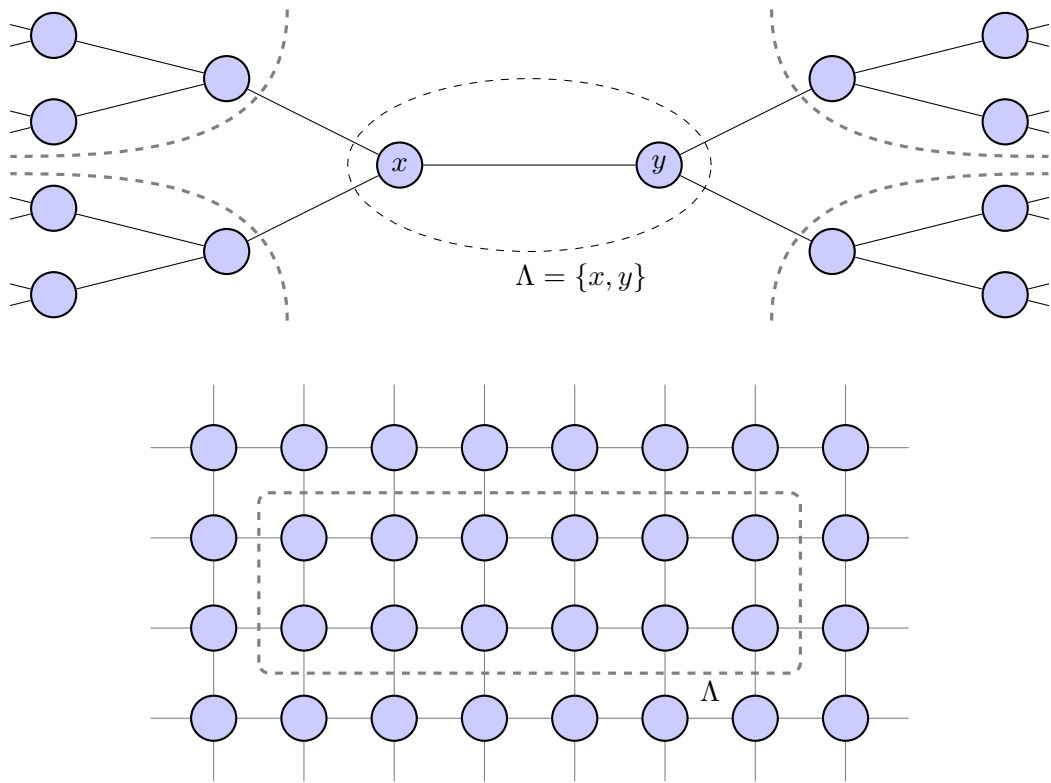
\begin{figure}
\centering
\begin{subfigure}{1\linewidth}
\centering
	\beginpgfgraphicnamed{tree-nature}
		\begin{tikzpicture}[scale=1.15]
		\clip (-9.5,-1.5) rectangle (2.5,3.5);
		
		\node (a11) at (-11,3) {};
		\node (a12) at (-11,2) {};
		
		\node (a21) at (-11,2) {};
		\node (a22) at (-11,1) {};
		
		\node (a31) at (-11,1) {};
		\node (a32) at (-11,0) {};
		
		\node (a41) at (-11,0) {};
		\node (a42) at (-11,-1) {};

		\node (b11) at (4,3) {};
		\node (b12) at (4,2) {};
		
		\node (b21) at (4,2) {};
		\node (b22) at (4,1) {};
		
		\node (b31) at (4,1) {};
		\node (b32) at (4,0) {};
		
		\node (b41) at (4,0) {};
		\node (b42) at (4,-1) {};
		
		\node[place] (w1) at (-9,2.5) {}
			edge (a11)
			edge (a12);
			
		\node[place] (w2) at (-9,1.5) {}
			edge (a21)
			edge (a22);
		
		\node[place] (x1) at (-9,0.5) {}
			edge (a31)
			edge (a32);	
	
		\node[place] (x2) at (-9,-0.5) {}
			edge (a41)
			edge (a42);
		
		\node[place] (u1) at (2,2.5) {}
			edge (b11)
			edge (b12);
			
		\node[place] (u2) at (2,1.5) {}
			edge (b21)
			edge (b22);
		
		\node[place] (v1) at (2,0.5) {}
			edge (b31)
			edge (b32);
		
		\node[place] (v2) at (2,-0.5) {}
			edge (b41)
			edge (b42);
		
		\node[place] (k1) at (0,2) {}
			edge 		(u1)
			edge 			(u2);
			
		\node[place] (k2) at (0,0)  {}
			edge  		(v1)
			edge 			(v2);
		
		\node[place] (m1) at (-7,2) {}
			edge  		(w1)
			edge 			(w2);
			
		\node[place] (m2) at (-7,0)  {}
			edge  		(x1)
			edge 			(x2);
			
		\node[place] (j) at (-5,1)  {$x$}
			edge  		(m1)
			edge 			(m2);
		
		\node[place] (i) at (-2,1)  {$y$}
			edge  		(k1)
			edge 			(k2)
			edge 		(j);
			
		\node at (-2.9,-0.3) {$\L=\{x,y\}$};
			
		\draw[dashed] (-3.5,1) ellipse (2.1cm and 1cm);
		
		  \draw[very thick, gray, dashed] 
    			(-0.7,2.8).. controls +(down:1.6cm) and +(left:1.5cm) ..(2.5,1.1);
			
 		  \draw[very thick, gray, dashed] 
    			(-0.7,-0.8).. controls +(up:1.6cm) and +(left:1.5cm) ..(2.5,0.9);
		
		  \draw[very thick, gray, dashed] 
    			(-6.3,-0.8).. controls +(up:1.6cm) and +(right:1.5cm) ..(-9.5,0.9);
			
		 \draw[very thick, gray, dashed] 
    			(-6.3,2.8).. controls +(down:1.6cm) and +(right:1.5cm) ..(-9.5,1.1);
		
		\end{tikzpicture}
	\endpgfgraphicnamed
	
\end{subfigure}

\begin{subfigure}{1\linewidth}
\centering
	\beginpgfgraphicnamed{grid}
		\begin{tikzpicture}[scale=1.2]
			\draw[step=1cm,color=gray] (-5.7,-2.7) grid (2.7,1.7);
			
			\draw[very thick, gray, dashed, rounded corners] (-4.5,-1.5) rectangle (1.5,0.5); 
				
			\foreach \x in {-5,-4,-3,-2,-1,0,1,2} {
				\foreach \y in {-2,-1,0,1} 
					\node[place] at (\x,\y) {};
			};
			
			\node at (0.5,-1.7) {$\L$};
			
		\end{tikzpicture}
	\endpgfgraphicnamed
	
\end{subfigure}

\caption{For the binary tree the complement of the finite set $\L\subset \subset V$ consists of four infinite disconnected subtrees. Hence the relative boundary height is not recovered by the gradient information among edges contained in these infinite subtrees. The complement of the finite set $\L\subset \subset \Z^2$ consists however of exactly one connected subgraph.}

\label{fig:grid}
\end{figure}

\begin{defn}
Let $\L \subset V$. For any two infinite-volume gradient configurations $\rho, \zeta \in \Omega^\nabla$ we write $\rho \sim_{\partial \L} \zeta$, if there exist two height configurations $\varphi, \psi \in \Omega$ s.t. $\nabla \varphi = \rho$, $\nabla \psi = \zeta$ and $\varphi_{\partial \L} = \psi_{\partial \L}$. 
\end{defn}

We note that $\rho \sim_{\partial \L} \zeta$ if and only if 
$$\sum_{b \in \Gamma(x,y)} \rho_b = \sum_{b \in \Gamma(x,y)} \zeta_b$$
for any vertices $x,y \in \partial \L$, where $\Gamma(x,y) \subset \vec E$ denotes the unique self-avoiding path from $x$ to $y$.
Suppose that $\L$ is any fixed 
finite subtree. Then the above property 
can be decided when we know both gradient configurations $\eta,\zeta$ 
for the edges with sites on $\L \cup \partial \L$, that is on the inside. 
Indeed, this allows us to reconstruct 
 the relative heights at the sites on $\partial \L$ in both configurations. 

We define a map $\zeta \mapsto [\zeta_{\partial \L}]$ with values in these 
equivalence classes, that is taking values in $\Z^{\partial \L}/ \Z$. 
Observe that 
this map is then measurable w.r.t. to the sigma-algebra which is generated by the gradient variables inside $\L\cup\partial\L$, i.e. $\s( (\eta_b)_{b \subset (\L \cup \partial \L)})$.  
However it is not measurable w.r.t. the naive outer 
sigma-algebra 
$\mathcal{F}^\nabla_{\L^c}$, 
where we put 
$\mathcal{F}^\nabla_{W^c}:=\s( (\eta_b)_{b \cap W= \emptyset})$
for any (possibly infinite) subset $W\subset V$ . 
This would only be the case for any graphs with the property 
that after subtraction of any finite subvolume the graph is still connected, like it is the case for lattices in more than one dimension (see Figure \ref{fig:grid}). 

Therefore, for the desired kernels $\gamma'_{\L}$ on the gradient space
the sigma-algebras to be considered need to keep this information, 
and be larger than just the product sigma-algebras over the gradient variables 
with bonds in the outside. 

\begin{defn}\label{outside} 
Let $\L\subset\subset V$ be any finite subvolume. Then the {\em gradient sigma-algebra 
outside $\L$} is defined to be  
$$\mathcal{T}^\nabla_{\L}=\s( (\eta_b)_{b \cap \L = \emptyset}, \left[\eta_{\partial \L}\right]) \subset \mathcal{F^\nabla}$$
\end{defn}

\begin{rem} 
We note that  $\L_2 \supset \L_1$ implies $\mathcal{T}_{\L_2}\subset \mathcal{T}_{\L_1}$.  
This is true since the information from the boundary condition in the smaller volume $\L_1$ 
supplemented with information from the annulus $\L_2 \backslash \L_1$ 
allows to recover the 
information of the boundary condition in the larger volume. 
\end{rem}

\begin{rem} 
Observe that $\mathcal{F}^\nabla_{\L^c}$ is strictly smaller than 
$\mathcal{T}^\nabla_{\L}$. This is quite unusual compared 
to the Gibbsian setup for lattice spin systems, and gradient systems (in two or more dimensions). 
\end{rem}

Due to the tree nature, the usual plaquette condition for gradient configurations (saying 
that walking around circles on the graph 
we arrive at the same initial height) is empty for configurations 
inside a given volume. 
However, this does not make our specification trivial, as the 
relative-height constraint remains on the boundary of the volume. 

The gradient specification we define will be the restriction of the previously 
defined Gibbs specification to the smaller sigma-algebra $\FF^{\nabla}\subset \FF$.  
Let $B\in \FF^{\nabla}$ a height-shift invariant set. 
Note that the Gibbsian probability 
$\gamma_{\L}(B \mid \cdot)$ is measurable w.r.t. to $\mathcal{T}_{\L}$, but not w.r.t. the smaller 
sigma-algebra generated by the gradient variables outside of $\L$ which does not 
contain the relative-height information on the boundary.

\begin{defn}
The {\em gradient Gibbs specification} is defined as the family of probability kernels $(\gamma'_\L)_{\L \subset \subset V}$ from $(\Omega^\nabla, \mathcal{T}_\L)$ to $(\Omega^\nabla, \mathcal{F}^\nabla)$ such that
\begin{equation}\begin{split}\label{grad}
\int F(\rho) \gamma'_\L(d\rho \mid \zeta) = \int F(\nabla \varphi) \gamma_\L(d\varphi\mid\omega)
\end{split}
\end{equation}
for all bounded $\FF^{\nabla}$-measurable functions $F$, where $\omega \in \Omega$ is any height-configuration with $\nabla \omega = \zeta$.
\end{defn}


A more explicit writing goes like this, using the tree property.   
Let $\a^Q_{\L\cup\partial\L}$ denote the product specification on the bonds $b$ inside $\L\cup\partial\L$
given by the transfer operator $Q_b(\cdot)$, i.e.
\begin{equation}
\a_{\L\cup\partial\L}^Q(\rho_{\L\cup\partial\L}) = Z_\L^{-1}(\rho) \prod_{b\cap\L\neq\emptyset} Q_b(\rho_b) 
\end{equation}
where $Z_\L(\rho)$ is a normalizing constant. 
Then the l.h.s. of \eqref{grad} is given by
\begin{equation}\begin{split}
\gamma'_{\L}(F \mid \zeta)=\frac{
\sum_{\rho_{\L\cup \partial \L}}\a^Q_{\L\cup \partial \L}(
\rho_{\L\cup \partial \L}) F(\rho_{\L\cup \partial \L} \zeta_{\L^c})
\mathbf{1}_{[\rho_{\partial\L}]=[\zeta_{\partial\L}]}}
{
\sum_{\rho_{\L\cup \partial \L}}\a^Q_{\L\cup \partial \L}(
\rho_{\L\cup \partial \L})
\mathbf{1}_{[\rho_{\partial\L}]=[\zeta_{\partial\L}]}
}.
\end{split}
\end{equation}
In the concatenation $\eta_{\L\cup \partial \L} \zeta_{\L^c}$ 
the subscripts denote that gradient configurations should be taken 
with both endpoints of edges on the indicated sets of sites. 

Using the outer sigma-algebra 
 $\mathcal{T}_{\L}$, this is now a proper and consistent family of probability kernels, i.e. 
\begin{equation}
\gamma'_\L(A \mid \zeta) = \mathbf{1}_A(\zeta)
\end{equation} 
for every $A \in \mathcal{T}_\L$ and $\gamma'_\Delta \gamma'_\L = \gamma'_\Delta$ for any finite volumes $\L, \Delta \subset V$ with $\L \subset \Delta$. The proof is similar to the situation of regular local Gibbs specifications \cite[Proposition 2.5]{Ge88}. 

Let $\mathcal{C}_b(\Omega^\nabla)$ be the set of bounded functions on $\Omega^\nabla$. 
Gradient Gibbs measures will now be defined in the usual way by having its conditional probabilities outside finite regions prescribed by the gradient Gibbs specification:

\begin{defn} A measure $\nu \in \mathcal{M}_1(\O^\nabla)$ is called a {\em gradient Gibbs measure (GGM)} if it satisfies the DLR equation 
\begin{equation}\label{DLR}
\int \nu (d\zeta)F(\zeta)=\int \nu (d\zeta) \int \gamma'_{\L}(d\tilde\zeta \mid \zeta) F(\tilde\zeta)
\end{equation}
for every finite $\Lambda \subset V$ and for all $F \in \mathcal{C}_b(\Omega^\nabla)$. The set of gradient Gibbs measures will be denoted by $\GG^\nabla(\g)$ or $\GG^\nabla(Q)$.  
\end{defn}

The advantage of gradient Gibbs measures is that they may exist, even in situations where a proper Gibbs measure does not. An example for this is the massless discrete Gaussian free field on the lattice $\Z^d$ in dimensions $d\leq 2$. Let $B(n)$ be the box of width $n$ and $\psi$ any boundary condition along its boundary. Then it can be shown that the variance of the height variable at $0$ under the finite-volume Gibbs measure $\gamma_{B(n)}(\cdot \mid \psi)$ goes to $\infty$ for $n\to\infty$. The field is said to \textit{delocalize} \cite{Ve06}. 

In the following we will work towards a representation of gradient Gibbs measures via the notion of so-called boundary laws \cite{Cox77}, \cite{Ge88}, \cite{Z83}:

\begin{defn}\label{def:bl}
A family of vectors $\{ l_{xy} \}_{\langle x,y \rangle \in \vec E}$ with $l_{xy} \in (0, \infty)^\Z$ is called a {\em boundary law for the transfer operators $\{ Q_b\}_{b \in E}$} if for each $\langle x,y \rangle \in \vec E$ there exists a constant  $c_{xy}>0$ such that the consistency equation
\begin{equation}\label{eq:bl}
l_{xy}(\omega_x) = c_{xy} \prod_{z \in \partial x \setminus \{y \}} \sum_{\psi_z \in \Z} Q_{zx}(\omega_x-\psi_z) l_{zx}(\psi_z)
\end{equation}
holds for every $\omega_x \in \Z$. A boundary law is called to be {\em $q$-periodic} if $l_{xy} (\omega_x + q) = l_{xy}(\omega_x)$ for every oriented edge $\langle x,y \rangle \in \vec E$ and each $\omega_x \in \Z$. 

\end{defn}

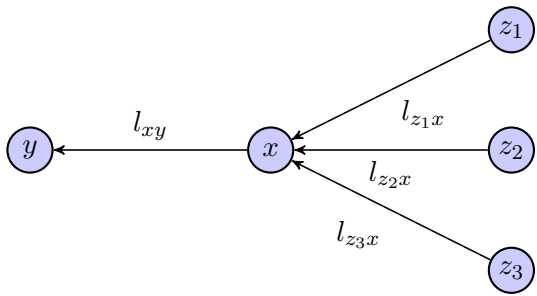
\begin{figure}
	\beginpgfgraphicnamed{boundary-law}
		\begin{tikzpicture}[scale=1.6]
		\node[place] (z1) at (0,2) {$z_1$};
		\node[place] (z2) at (0,1)  {$z_2$};
		\node[place] (z3) at (0,0)  {$z_3$};
		\node[place] (y) at (-4,1)  {$y$};
	
		\node[place] (x) at (-2,1)  {$x$}
			edge [pre] 	node[auto,swap] {$l_{z_1 x}$}	(z1)
			edge [pre] 	node[auto,swap] {$l_{z_2 x}$}	(z2)
			edge [pre]		node[auto,swap] {$l_{z_3 x}$}	(z3)
			edge [post]	node[auto,swap] {$l_{xy}$}	(y);
		\end{tikzpicture}
	\endpgfgraphicnamed

	\caption{For the boundary law $l$ the value of $l_{xy}(\omega_x)$ along the oriented edge $\langle x,y \rangle$ can be recursively determined by the values of $\{ l_{z_i x} \}_{i=1,2,3}$ along the oriented edges pointing towards $x$ via equation \eqref{eq:bl}.}
	\label{fig:bl}
\end{figure}

For periodic boundary laws all appearing sums are finite, under our assumption $Q\in l^1(\Z)$.
Note that while the transfer operators $\{ Q_b \}_{b \in E}$ possess reflection symmetry in spin space, i.e. $Q_b(\omega_b) = Q_b(-\omega_b)$, this is not necessarily the case for the class of boundary laws.

\section{Construction of gradient measures via periodic boundary laws}

We want to remind the reader of the definition of a tree-indexed Markov chain for tree-indexed
Gibbs measures. 
 To formulate this we need some more notation. For any vertex $w\in V$ the set of the directed edges pointing away from $w$ is given by
$$\vec E_w = \{ \langle x, y \rangle \in E : d(w,y) = d(w,x)+1 \}.$$
This set can be interpreted as the ''future'' of the vertex $w$. 
Furthermore we define the ''past'' of any oriented edge $\langle x,y \rangle \in \vec E$ by
$$(-\infty, xy) = \{ w \in V \mid \langle x,y \rangle \in \vec E_w \}.$$

\begin{figure}
	\beginpgfgraphicnamed{pastedge}
		\begin{tikzpicture}[scale=1.]
		
		\clip (-10.5,-1) rectangle (-1,3);
		
		\node[place] (a1) at (-11,3) {};
		\node[place] (a2) at (-11,2) {};
		
		\node[place] (b1) at (-11,2) {};
		\node[place] (b2) at (-11,1) {};
		
		\node[place] (c1) at (-11,1) {};
		\node[place] (c2) at (-11,0) {};
		
		\node[place] (d1) at (-11,0) {};
		\node[place] (d2) at (-11,-1) {};
		
		\node[place] (w1) at (-9,2.5) {}
			edge [pre] 		(a1)
			edge [pre] 		(a2);
			
		\node[place] (w2) at (-9,1.5) {}
			edge [pre] 		(b1)
			edge [pre] 		(b2);
		
		\node[place] (x1) at (-9,0.5) {}
			edge [pre] 		(c1)
			edge [pre] 		(c2);
		
		\node[place] (x2) at (-9,-0.5) {}
			edge [pre] 		(d1)
			edge [pre] 		(d2);

		\node[place] (m1) at (-7,2) {}
			edge [pre] 		(w1)
			edge [pre] 			(w2);
			
		\node[place] (m2) at (-7,0)  {}
			edge  		[pre] 	(x1)
			edge 		[pre] 	(x2);
			
		\node[place] (k1) at (0,2) {};
		\node[place] (k2) at (0,0) {};
			
		\node[place] (x) at (-5,1)  {$x$}
			edge  	[pre] 		(m1)
			edge 	[pre] 		(m2);
		
		\node[place] (y) at (-2,1)  {$y$}
			edge  		(k1)
			edge 		(k2)
			edge [pre]		(x);
			
		\draw[dashed] (-3.5,-1) .. controls +(right:0.3cm) and +(right:0.3cm) .. (-3.5,3.0);
		
		\end{tikzpicture}
	\endpgfgraphicnamed

	\caption{The set of oriented edges $(-\infty, xy)$.}
	\label{fig:pastedge}
\end{figure}
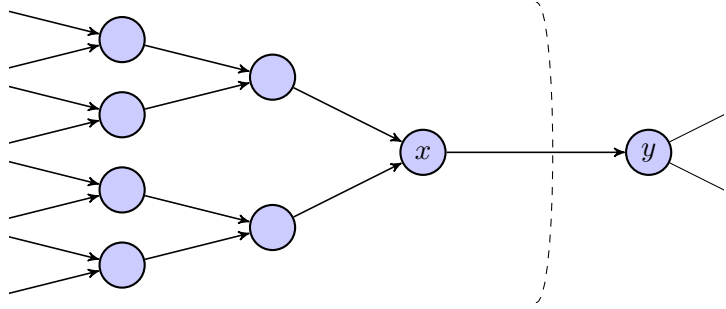

\begin{defn}
Let $\Omega_0$ be the local state space and $\Omega = \Omega_0^V$. 
A measure $\mu \in \mathcal{M}_1(\Omega)$ is called a {\em tree-indexed Markov chain} if
\begin{equation}\label{mc}
\mu(\phi_y = \omega_y \mid \mathcal{F}_{(-\infty, xy)}) = \mu(\phi_y = \omega_y \mid \mathcal{F}_{\{x\}}) 
\end{equation}
$\mu$-a.s. for any $\langle x,y \rangle \in \vec E$ and any $\omega_y \in \Omega_0$. 
\end{defn}

For finite local state spaces $\Omega_0$ it is well known that every Gibbs measure w.r.t. a specification of the form \eqref{spec} which is a Markov chain corresponds to a boundary law, which is unique up to a positive pre-factor \cite[Theorem 12.12]{Ge88}. Conversely, every boundary law $\{ l_{xy} \}_{\langle x,y \rangle \in \vec E}$ with $l_{\langle x,y \rangle}\in(0,\infty)^{\Omega_0}$ for each $\langle x,y \rangle \in \vec E$ defines a unique Markov chain $\mu \in \mathcal{M}_1(\O)$ in the set of Gibbs measures via the equation 
\begin{equation}\label{GeBL}
\mu(\phi_{\L\cup\partial\L} = \omega_{\L\cup\partial\L}) = \bar Z_\L^{-1} \prod_{y\in\partial\L} l_{yy_\L}(\omega_y) \prod_{b\cap\L\neq\emptyset} Q_b(\omega_b).
\end{equation}
Here $\bar Z_\L$ is a normalizing constant and $\L\subset V$ is any finite connected set. Note that if $\Lambda \subset V$ is a connected set and $y \in \partial \L$, then $\L \cap \partial y$ consists of a unique element which is denoted here by $y_\L$. 

A similar result can be obtained under the additional assumption that the boundary law is normalizable \cite[Theorem 3.2]{Z83}, i.e. 
\begin{equation}\label{eq:norm}
\sum_{\omega_x \in \Z} \left( \prod_{z \in \partial x} \sum_{\psi_z \in \Z} Q_{zx} (\omega_x - \psi_z) l_{zx} (\psi_z) \right) < \infty
\end{equation}
for each $x \in V$. Note that this condition is needed since it guarantees that $\bar Z_{\{i\}} < \infty$ for every $i \in V$ and hence by consistency \eqref{eq:bl} that $\bar Z_\L < \infty$ for every finite connected sub-volume $\L \subset V$. However, this assumption rules out many interesting cases like e.g. periodic boundary laws.


\begin{defn}
Let the {\em mod-$q$ fuzzy map} $T_q : \Z \to \Z_q$ be given by $T_q(i) = i\text{ mod } q$, where $\Z_q = \{ 0,...,q-1\}$ for $n\in\mathbb{N}$. 
\end{defn}

For any connected sub-volume $A\subset V$ let $\O^\nabla_A$ denote the set of gradient configurations on $A$, i.e. $\O^\nabla_A = \Z^{\vec E(A)}$, where $\vec E(A)$ are the directed edges connecting the vertices in $A$.
Now we define gradient measures in the infinite volume which are associated 
to a boundary law by pinning the spin at a given site $w\in V$ to take values in a given class (or layer).
More precisely we have the following theorem. 

We define a family of marginal measures in some analogy to the boundary law 
representation \eqref{GeBL} of \cite[Theorem 12.12]{Ge88}, but supplemented with internal information 
about layers. 

\begin{thm}\label{PinGGM} 
Let a vertex $w\in\Lambda$, where $\Lambda \subset V$ is any finite connected set, and a class label $s \in \Z_q$ be given. 
Then any $q$-periodic boundary law $\{ l_{xy} \}_{\langle x,y \rangle \in \vec E}$ for $\{ Q_b \}_{b \in E}$ defines a consistent family of probability measures on the gradient space $\Omega^\nabla$ by
\begin{equation}\label{bl1}
\nu_{w,s}(\eta_{\Lambda \cup \partial \Lambda}=\zeta_{\L\cup\partial\L})
= 
c_\Lambda(w,s) \prod_{y \in \partial \Lambda} l_{yy_\L}\left(\varphi'_y(s, \zeta)\right) \prod_{b \cap \Lambda \neq \emptyset}
Q_b(\zeta_b),
\end{equation}
where $\zeta_{\L\cup\partial\L}\in \Z^{\vec E(\Lambda \cup \partial \Lambda)}$. Here 
$$\varphi'_y(s , \zeta)=T_q\Bigl (
s+\sum_{b\in \Gamma(w,y)}\zeta_b
 \Bigr)$$ denotes the class in $\Z_q$ obtained by walking from class $s$ at the site $w\in \Lambda$ 
 along the unique path $\Gamma(w,y)$ to the boundary site $y$ whose class 
 is determined by the gradient configuration $\zeta$. 
 Since the boundary law is a class function, expression \eqref{bl1} is well-defined, where
$c_\Lambda(w,s)$ is a normalization factor that turns 
$\nu_{w,s}$ into a probability measure on $\Z^{\vec E(\Lambda \cup \partial \Lambda)}$. 
\end{thm}

\begin{proof}
The distributions in \eqref{bl1} are consistent if

\begin{equation}\begin{split}\label{cs}
\sum_{\zeta_{A} \in A^\nabla }  c_\Delta(w,s) \prod_{y \in \partial \Delta} & l_{yy_\L} (\sigma'_y(s, \zeta)) \prod_{b \cap \Delta \neq \emptyset} Q_b(\zeta_b) \\
&= c_\Lambda(w,s) \prod_{y \in \partial \Lambda} l_{yy_\L} (\sigma'_y(s, \zeta)) \prod_{b \cap \Lambda \neq \emptyset} Q_b(\zeta_b),
\end{split}\end{equation}
whenever $\Lambda, \Delta \subset V$ are any finite connected sets with $\Lambda \subset \Delta$, $A = (\Delta \cup \partial \Delta) \setminus (\Lambda \cup \partial \Lambda)$ and $\zeta \in (\Lambda \cup \partial \Lambda)^\nabla$. 
We show that this is the case for any $\Delta := \Lambda \cup \{ z \}$, where $z \in \partial \Lambda$ (see Figure \ref{fig:proofpinGGM}). The claim then follows by induction. 

\begin{figure}
	\beginpgfgraphicnamed{proofpinGGM}
		\begin{tikzpicture}[scale=1.15]
		
		\node[place] (r1) at (3,1) {};
		\node[place] (r2) at (2,2)  {};
		
		\node[place] (s1) at (3,-1) {};
		\node[place] (s2) at (2,-2)  {};
		
		\node[place] (y1) at (-2,2) {};
		\node[place] (y2) at (-3,1)  {};
		
		\node[place] (v1) at (-3,-1) {};
		\node[place] (v2) at (-2,-2)  {};
		
		\node[place] (k1) at (2,1) {}
			edge [pre] 		(r1)
			edge [pre]			(r2);

		\node[place] (k2) at (2,-1)  {}
			edge [pre] 		(s1)
			edge [pre]			(s2);

		\node[place] (m1) at (-2,1) {}
			edge [pre] 		(y1)
			edge [pre]			(y2);

		\node[place] (m2) at (-2,-1)  {}
			edge [pre] 		(v1)
			edge [pre]			(v2);

		\node[place] (x) at (-1,0)  {}
			edge  		(m1)
			edge 		(m2);
		
		\node[place] (u1) at (6.8,1) {$u_1$};
		\node[place] (u2) at (6.8,0)  {$u_2$};
		\node[place] (u3) at (6.8,-1)  {$u_3$};
		
		\node[place] (w) at (1,0)  {$z_\L$}
			edge 		(k1)
			edge 		(k2)
			edge 		(x);
		
		\node[place] (z) at (4,0) {$z$}
			edge [pre] (u1)
			edge [pre] (u2)
			edge [pre] (u3)
			edge [post] (w);

	\draw[dashed] (6.8,0) ellipse (0.8cm and 1.8cm);
		\node at (7.8,-0.5) {$A$};
	\draw[dashed,rounded corners,thick] (-2.5,-1.5) rectangle (2.5,1.5);
		\node at (-0.5,-1.3) {$\L$};
	\draw[dashed,rounded corners,thick] (-3.5,-2.5) rectangle (5.6,2.5);
		\node at (0.3,-2.3) {$\L\cup\partial\L$};
		
		\end{tikzpicture}
	\endpgfgraphicnamed

	\caption{The boundary law property guarantees the consistency of the family of marginals given by equation \eqref{bl1}.}
	\label{fig:proofpinGGM}
\end{figure}
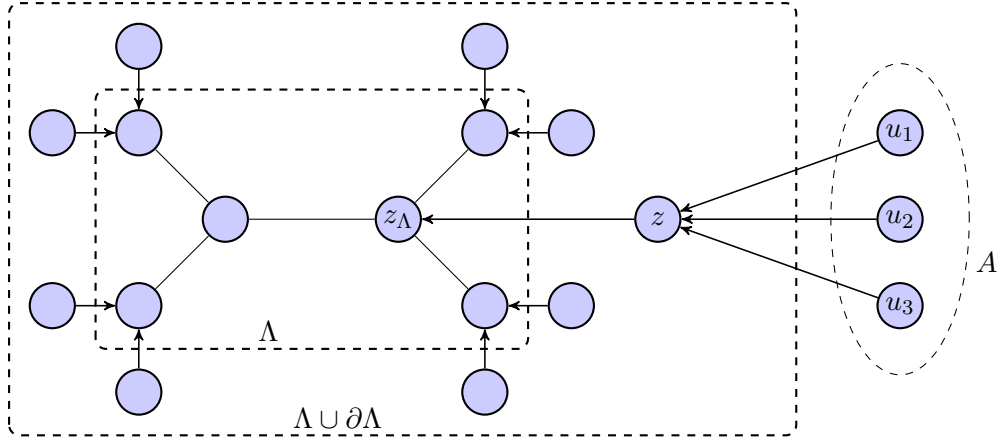

We have $A =\partial z \setminus \{ z_\L \}$.
From the definition of a boundary law we obtain
\begin{equation}\begin{split}\label{eq:cs}
& \sum_{\zeta_A \in \O_A^\nabla} c_\Delta(w,s) \prod_{y \in \partial \Delta} l_{yy_\L} \left( \varphi'_y(s, \zeta)\right)  \prod_{b \cap \Delta \neq \emptyset} Q_b(\zeta_b) \\
& = c_\Delta(w,s) \prod_{y \in \partial \Lambda \setminus \{ z \}} l_{yy_\L} \left(\varphi'_y(s, \zeta)\right) \prod_{b \cap \Lambda \neq \emptyset} Q_b(\zeta_b) \left( \prod_{u \in A} \sum_{\zeta_{uz} \in \Z} l_{uz} \left( \varphi_u'(s, \zeta)\right) Q_{uz} (\zeta_{uz}) \right) \\
& = c_\Delta(w,s) \frac{1}{c_{zz_\L}} \prod_{y \in \partial \Lambda} l_{yy_\L} \left(\varphi'_y(s, \zeta) \right) \prod_{b \cap \Lambda \neq \emptyset} Q_b(\zeta_b).
\end{split}\end{equation}
The second equality holds as $l$ is a $q$-periodic boundary law:
\begin{equation}\begin{split}
\prod_{u \in A} \sum_{\zeta_{uz} \in \Z} l_{uz} \left(\varphi_y'(s, \zeta)\right) Q_{uz} (\zeta_{uz}) 
&= \prod_{u \in A} \sum_{\zeta_{uz} \in \Z} l_{uz} \big( T_q(s + \sum_{b \in \Gamma(w,u)} \zeta_b)\big) Q_{uz} (\zeta_{uz}) \\
&= \prod_{u \in A} \sum_{\zeta_{uz} \in \Z} l_{uz} (s + \sum_{b \in \Gamma(w,u)} \zeta_b) Q_{uz} (\zeta_{uz} ) \\
&= \prod_{u \in A} \sum_{j \in \Z} l_{uz} (j) Q_{uz} (s+ \sum_{b \in \Gamma(w,z)} \zeta_b -j) \\
&= \frac{1}{c_{zz_\L}} l_{zz_\L}( s+ \sum_{b \in \Gamma(w,z)} \zeta_b)  \\
&= \frac{1}{c_{zz_\L}}  l_{zz_\L} (T_q(s+ \sum_{b \in \Gamma(w,z)} \zeta_b)).
\end{split}\end{equation}
The last expression in \eqref{eq:cs} equals the r.h.s of \eqref{cs} up to a factor of 
$\frac{c_\Delta(w,s)}{c_{zz_\L} \cdot c_\Lambda(w,s)}$. 
Summing over $\zeta_{\Lambda \cup \partial \Lambda}$ shows that this factor is $1$. Hence the distributions in \eqref{bl1} are consistent and from Kolmogorov's extension theorem follows that there exists a unique probability measure on the space of the gradient configurations $\Omega^\nabla$ with these exact marginals. 
\end{proof}

\begin{defn}
We call the measure $\nu_{w,s}$ with the marginals given by \eqref{bl1} a {\em pinned gradient 
measure} on the space of gradients $\O^\nabla$. 
\end{defn}

\begin{rem}
%
We will show later that these pinned gradient measures possess a  gradient \textit{Gibbs} property, however not for all volumes.   
Also, these measures will not be homogeneous w.r.t. tree automorphisms, in the same 
way as a Markov chain which is started in a fixed configuration achieves its homogeneity 
in time only asymptotically, for large times. 
Indeed, if we look at a local observable far away from the pinning site $w$ 
the pinning configuration $s\in\Z_q$ will be forgotten by the ergodic theorem for usual 
one-dimensional Markov chains, and the measure looks like an average over different 
pinning configurations.  
\end{rem}

Given a boundary law $\{ l_{xy}\}_{\langle x,y \rangle \in \vec E}$ we define an associated transition matrix by 
$$P_{x,y}(\omega_x,\omega_y)=\frac{Q_{yx}(\omega_y-\omega_x)l_{yx}(\omega_y)}{\sum_{\omega_y \in \Z} Q_{yx}(\omega_y - \omega_x)l_{yx}(\omega_y)}.
$$
If the boundary law has period $q$, then, taking into account this periodicity
 we can introduce the associated transition matrices $\bar P_{xy}:\Z_q\times \Z \mapsto [0,1]$ 
in the following way:
$$P_{x,y}(\omega_x,\omega_y)=:\bar P_{x,y} (T_q (\omega_x) , \omega_y-\omega_x).$$
In this notation $\bar P_{ x,y } (T_q (i), j-i)$ denotes the probability to see an height increase of $j-i$ along the edge $\langle x,y \rangle$ given the class $T_q(i)$ in the vertex $x$. How these matrices can now be used to describe the pinned gradient measures is stated in the next theorem.

\begin{thm}\label{Mar} 
Any pinned gradient measure $ \nu_{w,s}$, which is constructed via a boundary law as in \eqref{bl1} allows a representation of the form 
\begin{equation}
\nu_{w,s}(\eta_{\Lambda \cup \partial \Lambda}=\zeta_{\L\cup\partial\L})
=\prod_{\langle x,y\rangle \in \vec E_w: x,y \in \L \cup \partial \L}\bar P_{x,y}\Bigl(
T_q\bigl (s+\sum_{b\in \Gamma(w,x)}\zeta_b
 \bigr); 
\zeta_{\langle x,y \rangle}\Bigr),
\end{equation}
where $\L \subset V$ is any finite connected set and $\zeta_{\L\cup\partial\L} \in \Z^{\vec E(\L \cup \partial\L)}$.
\end{thm}

\begin{proof}
We fix any oriented bond $\langle u,v \rangle \in \vec E$, any increments $c, d \in \Z$ and $\omega \in \Omega^\nabla$. Let $\Lambda \subset V$ be any finite connected set s.t. $u \in \Lambda \subset (-\infty, \langle u, v \rangle)$, where 
$$(- \infty, \langle u,v \rangle) := \{ k \in V: \langle u,v \rangle \in \vec E_k \}$$
represents the ''past'' of the oriented edge $\langle u,v \rangle$.
We set $\Delta = \Lambda \cup \partial \Lambda \setminus \{ v \}$. From the representation we obtained in \eqref{bl1} follows
\begin{equation}\begin{split}
\frac{\nu_{w,s} (\eta_{\langle u,v \rangle} = c \mid \eta_\Delta = \omega_\Delta)}{\nu_{w,s} (\eta_{\langle u,v\rangle} = d \mid \eta_\Delta = \omega_\Delta)} 
= & \frac{\prod_{k \in \partial \Lambda} l_{kk_\L} (\varphi_k'(s, \omega) ) \prod_{b \cap \Lambda \neq \emptyset} Q_b(\omega_b)  }{\prod_{k \in \partial \Lambda} l_{kk_\L} (\varphi_k'(s, \omega)) \prod_{b \cap \Lambda \neq \emptyset} Q_b(\omega_b) } \\
= & \frac{  l_{vu} (\varphi_v'(s, \omega,c) ) Q_{vu} (c) }{ l_{vu} (\varphi_v'(s, \omega,d)) Q_{vu} (d)  },
\end{split}\end{equation}
where $\varphi_v'(s, \omega,c) = T_q( s + \sum_{b \in \Gamma(w,u)} \omega_b +c)$. 

Summing over $c \in \Z$ gives us 
\begin{equation}\begin{split}\label{ind}
\nu_{w,s} (\eta_{\langle u,v \rangle} = d \mid \eta_\Delta = \omega_\Delta) 
&= \frac{l_{vu} (T_q(s + \sum_{b \in \Gamma(w,u)} \omega_b + d)) Q_{vu}(d)}{\sum_{c \in \Z} 
l_{vu} (T_q(s + \sum_{b \in \Gamma(w,u)} \omega_b + c)) Q_{vu}(c)} \\
&= \frac{ l_{vu}( s + \sum_{b \in \Gamma(w,u)} \omega_b + d) Q_{vu}(d)}{\sum_{c \in \Z} 
l_{vu}(s + \sum_{b \in \Gamma(z,u)} \omega_b + c) Q_{vu}(c)} \\
&= P_{u,v} \Big(s+ \sum_{b \in \Gamma(w,u)} \omega_b, s+ \sum_{b \in \Gamma(w,u)} \omega_b + d\Big) \\
&= \bar{P}_{u,v} \Big(T_q (s + \sum_{b \in \Gamma(w,u)} \omega_b), d \Big).
\end{split}\end{equation}

Conditioning inductively from the inside to the outside then proves the claim: 
In the first step let $\Lambda = \{ w \}$. Furthermore, let $v_1,...,v_k$ be the children of $w$ and $e_1,...,e_k$ the oriented bonds connecting them to $w$, i.e. $e_i = \langle v_i, w \rangle$ for $i=1,...,k$. 
From the representation \eqref{bl1} follows
$$\nu_{w,s} (\eta_{e_1, ... ,e_k} = \zeta_1, ... , \zeta_k) = c_{e_1 , ... , e_k}(w,s) \prod_{i=1}^k l_{e_i}( \varphi_{v_i}'(s, \zeta_i)) Q_{e_i} (\zeta_i).$$
As $c_{e_1 , ... , e_k}(w,s) = \prod_{i=1}^k \sum_{\zeta_i \in \Z} l_{e_i}( \varphi_v'(s, \zeta_i)) Q_{e_i}(\zeta_i)$, we find that 
$$\nu_{w,s}(\eta_{e_1 , ... , e_k} = \zeta_1 , ... , \zeta_k) = \prod_{i=1}^k \bar{P}_{w,v_i} \left( T_q (s), \zeta_i \right).
$$
The induction step has been shown in \eqref{ind}. 
\end{proof}


\begin{rem}
The measure $\nu_{w,s}$ is in itself not a Markov chain on the set of gradient configurations. To see this let $\{t,u\}$ be the last bond in the path $\Gamma(w,u)$. Then it is seen from the statement of Theorem \ref{Mar} that $\nu_{w,s} (\eta_{\langle u,v \rangle} = d \mid \eta_\Delta = \omega_\Delta)$ depends on $\{\omega_b\}_{b\in \Gamma(w,u)}$ and not simply on $\omega_{\{t,u\}}$. This is clear as the knowledge of the width of the last step from some vertices $t$ to $u$ does not determine which layer one has reached in $u$. This information can only be recovered by the complete information about the spin increments along the unique path $\Gamma(w, u)$. 
On the other hand, the measure $\nu_{w,s}$ resembles a Markov chain, 
but it has an additional internal degree of freedom in a finite space that needs to be memorized.

\end{rem}

\section{Tree homogeneity by mixing over fuzzy classes}

For finite local state spaces $\Omega_0$ every Markov chain $\mu \in \mathcal{M}_1(\Omega_0^V)$ can be written in the following way: Let $(P_{x,y})_{\langle x,y \rangle \in \vec E}$ be the transition probabilities of $\mu$ and let $\alpha_w$ be the marginal distribution of $\mu$ at some vertex $w\in V$. Then
\begin{equation*}\begin{split}
\mu(\phi_{\Lambda}=\varphi_\L)=\alpha_w(\varphi_w)\prod_{\langle x,y\rangle \in \vec E_w: x,y \in \L}P_{x,y}(\varphi_x,\varphi_y)
\end{split}
\end{equation*}
for every finite connected set $\L$ and each $\varphi_\Lambda \in \Omega_0^\L$ (see \cite[Formula 12.4]{Ge88}).
If the transition matrices $P_{x,y}$ are homogeneous, i.e. independent of the bond, and $\alpha$ is the invariant distribution, then
the measure $\mu$ will possess all tree-symmetries, i.e. $\mu$ is invariant under all graph automorphisms of $V$.

What is an equivalent of this 
in our case of gradient measures, and how do we get tree-symmetries? 
As we will see in the following, if we pin at a vertex $w\in V$ and average over the fuzzy classes $s\in \Z_q$ according to the suitable measure, which is in fact the invariant distribution of the fuzzy transform,
we get full tree-invariance. \\

Let us assume that $Q_b = Q$ for all $b\in E$. Until now $T$ could have been any locally finite tree. From now on we will restrict ourselves to the case of the $d$-regular \textit{Cayley tree}, i.e. $|\partial x| = d+1$ for every $x \in V$. 


We call a vector $l \in (0,\infty)^\Z$ a (spatially homogeneous) boundary law if there exists a constant $c>0$ such that
the consistency equation 
\begin{equation}\label{bl12}
l(i) = c \left(\sum_{j \in \Z} Q(i-j) l(j) \right)^d
\end{equation}
is satisfied for every $i \in \Z$. 

Note that by assumption $l(i)=1$ for every $i\in \Z$ is always a solution. 
 Given such a homogeneous boundary law $l$ we get for the associated transition matrix 
$$P(i,j)=\frac{Q(i-j)l(j)}{\sum_{k \in \Z} Q(i-k)l(k)}.
$$
We recall that $\bar P:\Z_q\times \Z \mapsto [0,1]$ is then given by
$$P(i,j)=:\bar P (T_q (i), j-i).$$
Furthermore, let the fuzzy transform $T_q P : \Z_q \times \Z_q \to [0,1]$ be defined by
\begin{equation}\label{fuz}
T_q P(\bar i, \bar j) = \sum_{j : T_q (j) = \bar j} P(\bar i,j) =: P'(\bar i, \bar j)
\end{equation}
for all layers $\bar i, \bar j \in \Z_q$.

\begin{thm}\label{MainThm} 
Let $\Lambda \subset V$ be any finite connected set and let $w\in \Lambda$ be any vertex. 
Let $\a^{(l)} \in \mathcal{M}_1(\Z_q)$ denote
 the unique invariant distribution for the fuzzy transform 
$T_q P^l$ of the transition matrix $P^l$ corresponding to the $q$-periodic homogeneous boundary law $l$. 
Then the measure $\nu \in \mathcal{M}_1(\Omega^\nabla)$ with marginals given by 
\begin{equation}\begin{split}\label{drei}
\nu(\eta_{\Lambda}=\zeta_{\Lambda})
=\sum_{\varphi'_w \in \Z_q}
\alpha^{(l)}(\varphi'_w)\prod_{\langle x,y\rangle \in \overrightarrow{E}_w: x,y \in \L}\bar P_{x,y}\Bigl(
T_q\bigl (
\varphi'_w+\sum_{b\in \Gamma(w,x)}\zeta_b
 \bigr),
\zeta_{\langle x,y \rangle}\Bigr)
\end{split}
\end{equation}
defines a (spatially) homogeneous GGM. 
\end{thm}

\begin{rem}
The gradient measures we constructed are non-trivial linear combinations 
of gradient measures which are non-homogeneous w.r.t. tree automorphisms but obtained 
with an "initial condition" at a singled out site $w$. 
They are averaged quenched measures of tree-indexed Markov chains in 
periodic environment (what probabilists often 
would call "annealed measures"). 
To get spatial homogeneity the average $\a$ over the environment is chosen in the following way: 
 
If we choose an initial fuzzy configuration $\varphi'_w$ at the site $w$ according to $\alpha$, then 
the variables 
$T_q\bigl (
\varphi'_w+\sum_{b\in \Gamma(w,x)}\zeta_b
 \bigr)$ have the very same distribution $\alpha$. 
 Hence this mixing measure $\alpha$ must be the invariant distribution 
 for the fuzzy transform. 
\end{rem}

\begin{proof}
The proof consists of considering 
the non-normalized measure 
\begin{equation}
(i,j)\mapsto l(i) Q(i-j) l(j),
\end{equation}
with $i,j\in \Z$. 
In a situation of finite local state spaces that would be 
a measure which is normalizable to a probability measure. 
In our situation this may be the case, for rapidly decaying $l(i)$'s, but in interesting cases, and in particular 
for height-periodic boundary laws, this is certainly not the case. 
Nevertheless, by dividing out the height-period $q$ we obtain all desired
 objects, namely 
all relevant single-site probability measures and transition kernels in terms of simple expressions. 
 
The invariant distribution of the fuzzy transition matrix \eqref{fuz} is given by
\begin{equation}
\a( \bar i)= \frac{l(\bar i) \sum_{j \in \Z}Q(\bar i-j) l(j)}{\sum_{\bar k\in \Z_q} l(\bar k) \sum_{j\in\Z}Q(\bar k-j) l(j)}
\end{equation}
for $\bar i\in \Z_q$. 
It is readily verified that $\a$ is indeed invariant under the application of $P'$.
From this we have that reversibility for $\alpha$ and $\bar P$ in the following form holds:
\begin{equation}\begin{split}\label{homo}
 \a(\bar i)\bar{P} (\bar i, j-i)=
\frac{ l(i) Q(i-j) l(j)}{ \sum_{\bar k \in \Z_q} l(\bar k) \sum_{m\in\Z} Q(k-m) l(m)}
= \a(\bar j) \bar{P} (\bar j, i-j),
\end{split}
\end{equation}
where $i, j$ are any integers such that $T_q (i)=\bar i$, $T_q (j)=\bar j$. 
The formula describes an interplay (a reversibility) of invariant distribution on layers and 
layer-dependent transition matrix. 

The spatial homogeneity of the layer-averaged pinned gradient measures as given in \eqref{drei} can now be seen as follows:\\
For a volume consisting of two neighboring sites, the spatial homogeneity 
is just the previous reversibility formula \eqref{homo}. \\
For general finite subtrees $\L$ we use induction over the number of sites, 
with hypothesis: The r.h.s.  of \eqref{drei} yields the same expression 
for all pinning sites $w\in \L$. 
Consider now a larger volume $\L'=\L\cup\{v\}$, where $v$ is a site adjacent to $\L$. 
To see that $\nu(\eta_{\L'}=\zeta_{\L'})$ can be written in the form of a r.h.s. of \eqref{drei} 
with also the pinning site $v$ allowed we argue as follows. 
First use the induction hypothesis to write $\nu(\eta_{\L'}=\zeta_{\L'})$ in terms 
of the pinning site $v_{\L}$ (which we recall is the unique neighbor of $v$ such that $v_\Lambda \in \L$). 
Next use the reversibility equation \eqref{homo} to switch the pinning site to $v$:
\begin{equation}\begin{split}
 \nu(\eta_{\Lambda'}=\zeta_{\Lambda'})
=&\sum_{\varphi'_{v_\L} \in \Z_q}
\alpha^{(l)}(\varphi'_{v_\L})\prod_{\langle x,y\rangle \in \vec E_{v_\L}: x,y \in \L}\bar P_{x,y}\Bigl(
T_q\bigl (
\varphi'_{v_\L}+\sum_{b\in \Gamma(v_\L,x)} \zeta_b
 \bigr),
\zeta_{\langle x,y \rangle}\Bigr) \\
= & \sum_{\varphi'_{v_\L} \in \Z_q}
\prod_{\langle x,y\rangle \in \vec E_{v_\L}: x,y \in \L \setminus \langle v_\L, v \rangle}\bar P_{x,y}\Bigl(
T_q\bigl (
\varphi'_{v_\L}+\sum_{b\in \Gamma(v_\L,x)} \zeta_b
 \bigr),
\zeta_{\langle x,y \rangle}\Bigr) \\
& \qquad \times  \alpha^{(l)}(\varphi'_{v_\L})   \times \bar P_{v_\L, v} \Bigl( T_q(\varphi_{v_\L}'), \zeta_{\langle v_\L, v \rangle} \Bigr)
\\
=& \sum_{\varphi'_{v} \in \Z_q}
\prod_{\langle x,y\rangle \in \vec E_{v_\L}: x,y \in \L \setminus \langle v_\L, v \rangle}\bar P_{x,y}\Bigl(
T_q\bigl (
\varphi'_{v}+\sum_{b\in \Gamma(v,x)} \zeta_b
 \bigr),
\zeta_{\langle x,y \rangle}\Bigr) \\
& \qquad \times  \underbrace{   \alpha^{(l)}(\varphi'_{v} + \zeta_{\langle v, v_\L \rangle})  \times \bar P_{v_\L, v} \Bigl( T_q(\varphi_{v}' + \zeta_{\langle v_\L, v \rangle}), \zeta_{\langle v_\L, v \rangle} \Bigr)
}_{= \alpha^{(l)}(\varphi'_{v}) \times \bar P_{v, v_\L} \Bigl( T_q(\varphi_{v}'), \zeta_{\langle v, v_\L \rangle} \Bigr)
} \\
=& \sum_{\varphi'_{v} \in \Z_q}    \alpha^{(l)}(\varphi'_{v})
\prod_{\langle x,y\rangle \in \vec E_{v}: x,y \in \L }\bar P_{x,y}\Bigl(
T_q\bigl (
\varphi'_{v}+\sum_{b\in \Gamma(v,x)} \zeta_b
 \bigr),
\zeta_{\langle x,y \rangle}\Bigr).
\end{split}
\end{equation}

It remains to prove that $\nu$ is indeed a gradient Gibbs measure in the sense that it 
satisfies the DLR equation w.r.t. the gradient specification, i.e.
\begin{equation}\label{DLR2}
\nu(\cdot \mid \mathcal{T}_\L) = \gamma'_\L
\end{equation}
$\nu$-almost surely for all finite sub-volumes $\L \subset V$. At first, consider the pinned gradient measures $\nu_{w,s}$. As has been noted before these measures 
are \textit{no} gradient \textit{Gibbs} measures, but they do have a {\em restricted gradient property} which holds for conditional 
probabilities in volumes away from the pinning site. 
More precisely: 
Let $\zeta, \omega \in \Omega^\nabla$ be any two gradient configurations, $\L, \Delta \subset V$ any finite connected sets with $\L \subset \Delta$ and $w \in (\Delta\cup\partial\Delta)\setminus\L$. Then by using representation \eqref{bl1} we see that
\begin{equation}\begin{split}
 & \frac{\nu_{w,s}(\eta_{\L\cup\partial\L} = \zeta_{\L\cup\partial\L} \mid [\eta_{\partial\L}] = [\omega_{\partial\L}], \eta_{(\Delta\cup\partial\Delta)\setminus\L}) = \omega_{(\Delta\cup\partial\Delta)\setminus\L})}{\nu_{w,s}(\eta_{\L\cup\partial\L} = \omega_{\L\cup\partial\L} \mid [\eta_{\partial\L}] = [\omega_{\partial\L}], \eta_{(\Delta\cup\partial\Delta)\setminus\L}) = \omega_{(\Delta\cup\partial\Delta)\setminus\L})}  \\
& \qquad \qquad = \left( \prod_{b\cap\L\neq\emptyset} \frac{Q_b(\zeta_b)}{Q_b(\omega_b)} \right) \mathbf{1}_{\{[\zeta_{\partial\L}] = [\omega_{\partial\L}]\}},
\end{split}\end{equation}
where we have used cancellations of the boundary law terms due to combined information 
of relative height information $[\omega_{\partial\L}]$ {\em and} layer information 
due to the pinning to layer $s$ at site $w$ outside of $\L$ (see Figure \ref{fig:mainthm}). 
(Note that this cancellation of boundary law terms
can not be used for pinning vertices $w$ inside of $\L$.)

\begin{figure}
\centering
	
	\beginpgfgraphicnamed{mainthm}
	
		\begin{tikzpicture}[scale=1.35]
		
		\node[dot] (x) at (-1,0) {$$};
		
		\node[dot] (y) at (0,0)  {$$}
			edge			(x);
		
		\node[dot] (w) at +([shift={(45:1)}]y) {$$}
			edge  		(y);
		
		\node[dot] (w1) at +([shift={(75:1)}]w) {$$}
			edge  		(w);
		\node[dot] (w2) at +([shift={(15:1)}]w) {$$}
			edge  		(w);	
		\node at +([shift={(-0.1,-0.2)}]w2) {$u$};
		
		\node[dot] (w11) at +([shift={(105:1)}]w1) {$$}
			edge  		(w1);
		\node[dot] (w21) at +([shift={(45:1)}]w1) {$$}
			edge  		(w1);

		\node[dot] (w21) at +([shift={(45:1)}]w2) {$$}
			edge  		(w2);
		\node[dot] (w22) at +([shift={(-15:1)}]w2) {$$}
			edge  [post]		(w2);	
		\node at +([shift={(-0.1,-0.2)}]w22) {$w$};
		\node at +([shift={(-0.3,0.3)}]w22) {$\omega_{\langle w,u \rangle}$};
		
		\node[dot] (b) at +([shift={(315:1)}]y) {$$}
			edge  		(y);
		\node (bal) at +([shift={(-0.25,0)}]b) {$v_\L$};
	
		\node[dot] (b1) at +([shift={(285:1)}]b) {$$}
			edge  		(b);
		\node[dot] (b2) at +([shift={(345:1)}]b) {$$}
			edge  [post]		(b);	
		
		\node at +([shift={(-0.1,-0.2)}]b2) {$v$};
		\node at +([shift={(-0.3,0.35)}]b2) {$l_{vv_\L}$};
		
		\node[dot] (b11) at +([shift={(255:1)}]b1) {$$}
			edge  		(b1);
		\node[dot] (b21) at +([shift={(315:1)}]b1) {$$}
			edge  		(b1);	
			
		\node[dot] (b21) at +([shift={(375:1)}]b2) {$$}
			edge  		(b2);
		\node[dot] (b22) at +([shift={(315:1)}]b2) {$$}
			edge  		(b2);	
		
		\node[dot] (a) at +([shift={(135:1)}]x) {$$}
			edge  	(x);

		\node[dot] (a1) at +([shift={(105:1)}]a) {$$}
			edge  		(a);
		\node[dot] (a2) at +([shift={(165:1)}]a) {$$}
			edge  		(a);	

		\node[dot] (a11) at +([shift={(135:1)}]a1) {$$}
			edge  		(a1);
		\node[dot] (a21) at +([shift={(85:1)}]a1) {$$}
			edge  		(a1);	
			
		\node[dot] (a21) at +([shift={(195:1)}]a2) {$$}
			edge  		(a2);
		\node[dot] (a22) at +([shift={(135:1)}]a2) {$$}
			edge  		(a2);	
			
		\node[dot] (u) at +([shift={(225:1)}]x) {$$}
			edge  		(x);
		\node[dot] (u1) at +([shift={(195:1)}]u) {$$}
			edge  		(u);
			
		\node[dot] (u2) at +([shift={(255:1)}]u) {$$}
			edge  		(u);	
		
		\node[dot] (u11) at +([shift={(225:1)}]u1) {$$}
			edge  		(u1);
		\node[dot] (u21) at +([shift={(165:1)}]u1) {$$}
			edge  		(u1);	
			
		\node[dot] (u21) at +([shift={(225:1)}]u2) {$$}
			edge  		(u2);
		\node[dot] (u22) at +([shift={(285:1)}]u2) {$$}
			edge  		(u2);	
		
		\node (lambda) at (-0.5,-1.2) {$\L$}; 
		\node (delta) at (2.6,-2.6) {$\Delta\cup\partial\Delta$}; 
		\node (equiv) at (2.5,0) {$[\omega_{\partial\L}]$};

		\draw[very thick, gray] ([shift=(180:2.2cm)]u) arc (180:270:2.2cm);
		\draw[very thick, gray] ([shift=(90:2.2cm)]a) arc (90:180:2.2cm);
		\draw[very thick, gray] ([shift=(0:2.2cm)]w) arc (0:90:2.2cm);
		\draw[very thick, gray] ([shift=(270:2.2cm)]b) arc (270:360:2.2cm);
		
		\draw[very thick, gray] (-1.707, -2.907) -- (0.707,-2.907);
		\draw[very thick, gray] (-1.707, 2.907) -- (0.707,2.907);
		\draw[very thick, gray] (-3.907, -0.707) -- (-3.907,0.707);
		\draw[very thick, gray] (2.907, -0.707) -- (2.907,0.707);

		\draw[rounded corners,dotted, thick] (-2,-1) rectangle (1,1);
		
		\draw[dashed, gray, thick] (b2) .. controls +(right:5mm) and +(right:5mm)
				.. (w2);	
			
		\end{tikzpicture}
	\endpgfgraphicnamed	
	
	\caption{For pinned layer at vertex $w\in (\Delta\cup\partial\Delta) \setminus \L$ the relative height information $[\omega_{\partial \L}]$ along with the gradient configuration $\omega_{\langle w,u \rangle}$ allows to recover the layer at any vertex $v \in \partial\L$. Therefore the boundary law $l_{vv_\L}$ does not depend on the gradient configuration inside $\L$.}
	\label{fig:mainthm}
\end{figure}
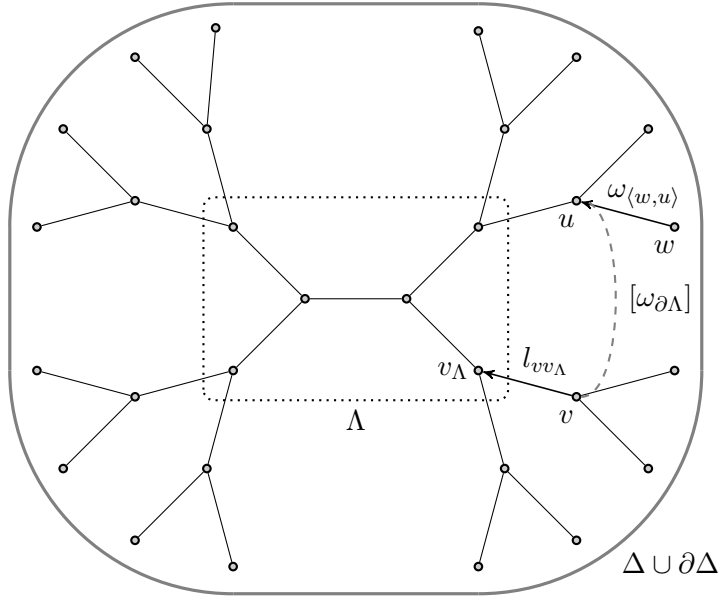

Summing over $\zeta_{\L\cup\partial\L}$ yields
\begin{equation}\begin{split}
& \nu_{w,s}(\eta_{\L\cup\partial\L} = \omega_{\L\cup\partial\L} \mid [\eta_{\partial\L}] = [\omega_{\partial\L}], \eta_{(\Delta\cup\partial\Delta)\setminus\L}) = \omega_{(\Delta\cup\partial\Delta)\setminus\L}) \\
&=  \frac{\prod_{b \cap \L \neq \emptyset} Q_b(\omega_b)}{\sum_{\zeta_{\L\cup\partial\L}} \left( \prod_{b\cap\L\neq\emptyset} Q_b(\zeta_b) \right) \mathbf{1}_{\{ [\zeta_{\partial\L}]=[\omega_{\partial \L}] \}}} \\
&=  \gamma'_\L(\omega_{\L\cup\partial\L} \mid [\omega_{\partial\L}],  \omega_{(\Delta\cup\partial\Delta)\setminus\L}).
\end{split}\end{equation}
Clearly 
\begin{equation}\begin{split}
&\nu_{w,s}(\eta_{\L\cup\partial\L}= \rho_{\L\cup\partial\L} \mid [\eta_{\partial\L}] = [\omega_{\partial\L}], \eta_{(\Delta\cup\partial\Delta)\setminus\L}=\omega_{(\Delta\cup\partial\Delta)\setminus\L}) \\
&\qquad \qquad = \gamma'_\L(\rho_{\L\cup\partial\L} \mid [\omega_{\partial\L}],  \omega_{(\Delta\cup\partial\Delta)\setminus\L}) =0
\end{split}\end{equation}
for any $\rho \in \Omega^\nabla$ with $[\rho_{\partial\L}] \neq [\omega_{\partial\L}]$ and hence the restricted gradient Gibbs property holds.

This property is now sufficient to prove that the spatially homogeneous gradient measure $\nu$ is indeed a GGM: Under the same assumptions as before we have
\begin{equation}\begin{split}
& \nu(\eta_{\L\cup\partial\L}= \rho_{\L\cup\partial\L} \mid [\eta_{\partial\L}] = [\omega_{\partial\L}], \eta_{(\Delta\cup\partial\Delta)\setminus\L}=\omega_{(\Delta\cup\partial\Delta)\setminus\L}) \\ 
&= \sum_{s \in \Z_q} \alpha(s) \nu_{w,s}(\eta_{\L\cup\partial\L}= \rho_{\L\cup\partial\L} \mid [\eta_{\partial\L}] = [\omega_{\partial\L}], \eta_{(\Delta\cup\partial\Delta)\setminus\L}=\omega_{(\Delta\cup\partial\Delta)\setminus\L})\\
&= \sum_{s \in \Z_q} \alpha(s) \gamma'_\L(\rho_{\L\cup\partial\L} \mid [\omega_{\partial\L}], \omega_{(\Delta\cup\partial\Delta)\setminus\L}) \\
&= \gamma'_\L(\rho_{\L\cup\partial\L} \mid [\omega_{\partial\L}], \omega_{(\Delta\cup\partial\Delta)\setminus\L}),
\end{split}\end{equation}
where we have switched the pinning vertex $w$ to some vertex outside $\L$ in the second step and then used the restricted gradient Gibbs property. 
Hence, the measure $\nu$ meets the DLR equation \eqref{DLR2} and therefore $\nu \in \GG^\nabla(\g)$.
\end{proof}

\begin{rem} 
 These GGMs also have a different representation which does not make explicit use of the invariant distribution $\a$ of the fuzzy transition matrix $P'$. Let $\L\subset V$ be any finite connected sub-volume and $w \in \L$ some pinning vertex. Then the measures given by \eqref{drei} can be written as
 \begin{equation}\label{alt}
 \nu (\eta_{\L\cup\L} = \zeta_{\L\cup\partial\L}) = c_\L \sum_{k\in\Z_q} \left( \prod_{y \in \partial\L} l_{yy_\L}\big(k + \sum_{b \in \Gamma(w,y)} \zeta_{b}\big) \prod_{b \cap \L \neq \emptyset} Q_b(\zeta_b) \right).
 \end{equation} 
Indeed, for a single-site volume $\L = \{ w \}$ representation \eqref{drei} of the GGM gives us
\begin{equation}\begin{split}
& \nu(\eta_{w\cup\partial w} = \zeta_{w\cup\partial w}) = \sum_{k\in \Z_q} \a(k) c(w,k) \prod_{y \in \partial i} l_{yw}(k + \zeta_{yw}) Q_{yw}(\zeta_{yw}) \\
&= \sum_{k\in\Z_q}  \left( \frac{l(k) \sum_{j \in \Z} Q(k-j)l(j)}{\sum_{m\in\Z_q}l(m) \sum_{j\in\Z}Q(m-j)l(j)} \left(\frac{1}{\sum_{j\in\Z} Q(k-j)l(j)}\right)^{d+1} \right. \\
 & \left. \qquad \quad \quad \times \prod_{y \in \partial i} l_{yw} (k+\zeta_{yw}) Q_{yw}(\zeta_{yw}) \right)\\
&= c_w \sum_{k \in \Z_q}  \prod_{y \in \partial w} l_{yw} (k+\zeta_{yw}) Q_{yw}(\zeta_{yw}),
\end{split}\end{equation}
where $c_w >0$ is a normalizing constant. In the second step we have used that $l$ is a boundary law and hence $l(k) / (\sum_{j \in \Z} Q(k-j)l(j))^d$ is a constant. For larger sub-volumes $\L$ the statement follows from a simple induction argument by using again the consistency property of boundary laws \eqref{eq:bl}.  

As mentioned in words before, we can also get the tree-automorphism invariant measure $\nu$ 
as a weak limit when we send the pinning vertex $w$ to infinity along any path, that is 
$\nu=\lim_{w \uparrow \infty}\tau_w \nu_{w,s}$ for any class $s$, where 
$\tau_w$ is the shift of the pinning point. 
\end{rem}

 \section{Applications}
 
 \subsection{Examples of GGMs for $q$-periodic boundary laws with $q\leq4$}
 A spatially homogeneous boundary law $l \in (0,\infty)^\Z$ on the Cayley tree of degree $d$ has to satisfy the recursive relation equation \eqref{bl12}. If $l$ is assumed to be $q$-periodic there exist positive constants $a_0, a_1,...,a_{q-1}$, s.t. 
 $l(i) = a_k$ if $i \text{ mod } q = k$. Note that we can always normalize a boundary law s.t. $a_0=1$.
Under this assumption \eqref{bl12} reduces to
 \begin{equation}\begin{split}
 a_k &= c \left( a_k Q(0) + \sum_{m=1}^q \left( (a_{k+m} + a_{k-m}) \sum_{j=0}^\infty Q(qj+m) \right) \right)^d
 \end{split}\end{equation}
 for any $k\in\Z_q$.
In the simplest possible non-trivial case of a $2$-periodic boundary law, i.e. $l(i) = 1$ for $i \text{ mod }2=0$ and $l(i) = a$ for $i \text{ mod }2=1$ with $a>0$, the system of b.l. equations becomes
\begin{equation}\begin{split}\label{ex1}
1 &= c \left( Q(0) + 2 \sum_{j=1}^\infty Q(2j) + 2a \sum_{j=0}^\infty Q(2j+1) \right)^d, \\
a &= c \left( 2 \sum_{j=0}^\infty Q(2j+1) +a \left(Q(0) + 2 \sum_{j=1}^\infty Q(2j)\right) \right)^d.
\end{split}\end{equation}  
Let us assume that we have a binary tree, i.e. $d=2$. 
Introducing $u = \sqrt{a}$ we have 
\begin{equation}
u = \frac{u^2 + 2u^2\sum_{j=1}^\infty Q(2j) + 2 \sum_{j=0}^\infty Q(2j+1)}{1 + 2 \sum_{j=1}^\infty Q(2j) + 2u^2 \sum_{j=0}^\infty Q(2j+1)}.
\end{equation}
In the case of the SOS-model, i.e. $Q_\beta(i,j) = e^{-\beta |i-j|}$ with $\b\in(0,\infty)$ being a free parameter (inverse temperature), we are left to solve the equation
\begin{equation}\begin{split}
\frac{1}{\sinh(\b)} (u-1) \left( u^2 + (1-\cosh(\b)) u + 1 \right)
=0.
\end{split}\end{equation}
Apart from the trivial solution $u_1=1$ the additional non-negative solutions
\begin{equation}
u_{2,3} = \frac{\cosh(\b)-1}{2} \pm \frac{\sqrt{\cosh^2(\b) -2\cosh(\b)-3}}{2}
\end{equation}
appear for $\cosh(\b) \geq 3$. Note that the corresponding boundary laws clearly do not meet the normalization requirement by Zachary \eqref{eq:norm}. 

A more indirect approach for a Cayley tree of arbitrary degree is described in the following. Note that the solutions we have found correspond to Ising-type boundary laws with some suitable interaction potential $\tilde U$: In the regular Ising model with state space $\Omega = \{-1,+1\}^V$ the b.l.'s are of the form $l=(1,a), a >0$, and the consistency equation boils down to
\begin{equation}\label{ising}
a = \left( \frac{\tilde Q(-,+) + a \tilde Q(+,+)}{\tilde Q(-,-)+ a \tilde Q(+,-)} \right)^d,
\end{equation}
where 
$$\tilde Q(-,-) = \tilde Q(+,+)= \exp(-\tilde U (+,+))$$
and 
$$\tilde Q(-,+) = \tilde Q(+,-) = \exp(- \tilde U(+,-)).$$
Compare to \cite[Formula 12.21]{Ge88}.
The second equation of \eqref{ex1} is therefore of Ising-type with 
\begin{equation}\begin{split}
- \tilde U(+,+) &= \log \tilde Q(+,+) = \log \left( Q(0) + \sum_{j=1}^\infty Q(2j) \right) = - \log \tanh(\b)
\end{split}\end{equation}
and 
\begin{equation}\begin{split}
- \tilde U(+,-) &= \log \tilde Q(+,-) = \log \left( 2 \sum_{j=0}^\infty Q(2j+1) \right) = - \log \sinh \b.
\end{split}\end{equation}
By adding a suitable constant $c \in \R$ to $\tilde U$ this interaction potential can be chosen to have the usual form
$$\tilde U (\s_i, \s_j) = - \tilde \b \s_i \s_j + c.$$
Indeed, in this case 
\begin{equation}\begin{split}
- \tilde U(+,+) &= - \log \tanh(\b) = \tilde \b - c, \\
- \tilde U(+,-) &= - \log \sinh(\b) = - \tilde \b - c,
\end{split}\end{equation}
which leads to
\begin{equation}\begin{split}
c(\b) &=  \frac{1}{2} \log \frac{\sinh^2 (\b)}{\cosh(\b)}, \\
\tilde \b(\b) &= \frac{1}{2} \log \cosh(\b).
\end{split}\end{equation}
For the Ising model on the Cayley tree of order $d$ it is well known that the critical value for the inverse temperature is given by $\tilde \b_{c,d}= \coth^{-1}(d)$, i.e. there exist multiple solutions to the b.l. equation if and only if $\tilde \b > \tilde \b_{c,d}$ \cite{Ly89}. 
As $\b \mapsto \tilde \b(\b)$ is a positive and monotone increasing function for $\b \in (0,\infty)$ with $\lim_{\b \to \infty} \tilde \b(\b) = \infty$ there is also a critical value $\b_{c,d}$ for the existence of multiple solutions to the gradient b.l. equations, namely $\b_{c,d} = \cosh^{-1}( \exp(2 \coth^{-1}(d))) = \cosh^{-1}((d+1)/(d-1))$. For $d=2$ the critical value is $\b_{c,2} = \cosh^{-1} (3)$, which confirms our earlier result. 


For $q=3$ the b.l. equations for the $q$-periodic gradient situation are
\begin{equation}\begin{split}
a &= \left( \frac{a (Q(0) + 2 \sum_{j=1}^\infty Q(3j) ) + (1+b) (\sum_{j=0}^\infty Q(3j+1) + \sum_{j=0}^\infty Q(3j+2)) }{ Q(0) + 2 \sum_{j=1}^\infty Q(3j)  + (a+b) (\sum_{j=0}^\infty Q(3j+1) + \sum_{j=0}^\infty Q(3j+2))} \right)^d \\
&= \left( \frac{a \left( 1+ \frac{2}{e^{3\b}-1}\right) + (1+b) \left( \frac{\cosh(\b/2)}{\sinh(3\b/2)} \right) }{ 1+  \frac{2}{e^{3\b}-1}  +  (a+b) \frac{\cosh(\b/2)}{\sinh(3\b/2)} } \right)^d
\end{split}\end{equation}
and 
\begin{equation}\begin{split}
b &= \left( \frac{b (Q(0) + 2 \sum_{j=1}^\infty Q(3j) ) + (1+a) (\sum_{j=0}^\infty Q(3j+1) + \sum_{j=0}^\infty Q(3j+2)) }{ Q(0) + 2 \sum_{j=1}^\infty Q(3j)  + (a+b) (\sum_{j=0}^\infty Q(3j+1) + \sum_{j=0}^\infty Q(3j+2))} \right)^d \\
&=\left( \frac{b \left( 1+ \frac{2}{e^{3\b}-1}\right) + (1+a) \left( \frac{\cosh(\b/2)}{\sinh(3\b/2)} \right) }{ 1+  \frac{2}{e^{3\b}-1}  +  (a+b) \frac{\cosh(\b/2)}{\sinh(3\b/2)} } \right)^d
\end{split}\end{equation}
For the regular Potts model with $q=3$ with an interaction potential $\tilde U(\s_i,\s_j) = -\tilde \b \mathbf{1}_{\s_i = \s_j} - c$, $c \in \R$, and boundary law $l=(1,a,b)$ the b.l. equations are of a similar form
\begin{equation}
a = \left( \frac{a e^{\tilde \b +c} + (1+b)e^c}{e^{\tilde \b + c} + (a+b) e^c} \right)^d
\end{equation}
and 
\begin{equation}
b = \left(  \frac{b e^{\tilde \b +c} + (1+a)e^c}{e^{\tilde \b + c} + (a+b) e^c} \right)^d.
\end{equation}
Hence,
\begin{equation}\begin{split}
e^{\tilde \b + c} &= 1 + \frac{2}{e^{3\b}-1} = \frac{\cosh(3\b/2)}{\sinh(3 \b /2)}, \\
e^c &= \frac{\cosh(\b/2)}{\sinh(3\b/2)},
\end{split}\end{equation}
which gives us 
\begin{equation}\begin{split}
\tilde\b(\b) &= \log \frac{\cosh(3\b/2)}{\cosh(\b/2)} = \log (2 \cosh(\b) - 1).
\end{split}\end{equation}
For the binary tree the critical value for the inverse temperature is known to be $\tilde\b_{c,2} = \log(1+2\sqrt{2})$ \cite{Ro}. As the map $\beta \mapsto \tilde\beta(\b)$ is again non-negative and monotone increasing with $\lim_{\b\to\infty} \tilde\b(\b) = \infty$, there also exists a critical value $\b_{c,2}$ for the existence of multiple solutions of the gradient b.l. equations, which is given by
\begin{equation}
\log(2 \cosh(\b_{c,2}) - 1) = \log (1+2\sqrt{2}) \Leftrightarrow \b_{c,2} = \cosh^{-1}(1+\sqrt{2}).
\end{equation}

For $q\geq 4$ however it is not difficult to see that the gradient b.l. equations are in general no longer of Potts-type. However let $l$ be a $4$-periodic b.l. of the particular form
$$  l(i) = \begin{cases}    
     1, & \text{for } i\text{ mod } q \in \{0,1\} \\
      a, & \text{for } i\text{ mod } q \in \{2,3\}.
        \end{cases}$$
Then the gradient b.l. equations turn into
\begin{equation}\begin{split}
a &=  \left( \frac{\sum_{j=0}^\infty Q(2j+1) + 2 \sum_{j=0}^\infty Q(4j+2) + a \left( Q(0) + \sum_{j=0}^\infty Q(2j+1) + 2 \sum_{j=1}^\infty Q(4j) \right)}{ Q(0) + \sum_{j=0}^\infty Q(2j+1) + 2 \sum_{j=1}^\infty Q(4j)   + a \left( \sum_{j=0}^\infty Q(2j+1) + 2 \sum_{j=0}^\infty Q(4j+2)  \right)} \right)^d \\
&=  \left( \frac{  \frac{1}{2\sinh \b} + \frac{1}{\sinh (2\b)} + a \left( 1 + \frac{1}{2 \sinh \b} + e^{-2\b}\frac{1}{\sinh (2\b)} \right)  }{1 + \frac{1}{2 \sinh \b} + e^{-2\b}\frac{1}{\sinh (2\b)} + a \left(\frac{1}{2\sinh \b} + \frac{1}{\sinh (2\b)} \right)    } \right)^d.
\end{split}\end{equation}
Again, this is now a b.l. equation of Ising type with
\begin{equation}
- \tilde U(+, +) = \log \left( 1 + \frac{1}{\sinh \b} + e^{-2\b}\frac{1}{\sinh (2\b)} \right) = \tilde \b - c
\end{equation}
and
\begin{equation}
- \tilde U(+, -) = \log \left( \frac{1}{2\sinh \b} + \frac{1}{\sinh (2\b)} \right) = -\tilde \b - c. 
\end{equation}
Solving this system of equations for $\tilde\b$ leads to
\begin{equation}
\tilde \b (\b) = \frac{1}{2} \log \left( 2 \cosh(\b) -1 \right).
\end{equation}
It is easily seen that this map is non-negative and monotone increasing with \linebreak $\lim_{\b \to \infty} \tilde\b(\b) = \infty$. Once more we make use of the known critical value for the Ising model, i.e. $\tilde \b_{c,d} = \coth^{-1} (d)$. The critical value $\b_{c,d}$ for the existence of multiple solutions to the gradient b.l. equations is therefore given by 
$$\frac{1}{2} \log \left(2 \cosh(\b_{c,d}) -1 \right) =  \coth^{-1} (d),$$
which means that $\b_{c,d} = \cosh^{-1}\left( \frac{d}{d-1} \right)$.

Note that for $q\geq5$ non-trivial boundary laws of the type  
$$  l(i) = \begin{cases}    
     1, & \text{for } i\text{ mod } q \in \{0,..., \lfloor q/2\rfloor\} \\
      a, & \text{for } i\text{ mod } q \in \{ \lfloor q/2\rfloor +1, ..., q-1\}
        \end{cases}$$
 will not exist since the b.l. equations will be over-determined in this case. However we will see in the following section how to construct models which allow the existence of multiple $q$-periodic boundary laws with $q\geq 5$.

\subsection{Construction of solvable models for higher values of $q$}
In the example above we have studied the periodic gradient b.l. equations for a specific choice of $Q$. We have seen that for $q=2$ and $q=3$ these models can be solved rigorously as they turn into an Ising and a Potts model respectively modulo monotone rescaling. 
Conversely, for higher values of $q\in\N$ it is always possible to construct a transfer operator $Q$ such that the gradient b.l. equations for the $q$-periodic case are of the Potts-type:

Note that, for any allowed transfer operator $Q$ on $\Z$, 
the equation for $q$-periodic boundary laws is equivalent 
to an equation for a clock-model (that is a $\Z_q$-invariant model on $\Z_q$). Indeed

\begin{equation}\begin{split}
a_i &= c \left(\sum_{j \in \Z}  Q(i-j) a_j \right)^d = c \left(\sum_{\bar j \in \Z_q}  T_q Q(i- \bar j) a_{\bar j} \right)^d \\
\end{split}\end{equation}
for every $i \in \Z_q$, where $T_qQ : \Z \to \R$ is given by $T_q Q(m) = \sum_{j\in\Z} Q(qj+m)$. 

Note also the obvious fact that the fuzzy-transformed function $m \mapsto T_q Q(m)$ describing 
the transition operator of the finite-dimensional model inherits the symmetry w.r.t. 
reflections of the spin-difference from the original model, i.e.
$ T_q Q (m)= T_q Q (-m)$. Equivalently, we can say that the transfer operator 
described by $T_q Q$ is a circulant matrix with additional reflection symmetry. 

Hence we see that the dimension of the set of possible $T_q Q$'s modulo 
constant is $D(q)= \lfloor q/2 \rfloor$.  
In particular, $D(2)=1$ and so any model on $\Z$ 
is mapped to an Ising-model with an effective inverse temperature.
Similarly $D(3)=1$,  and so we are back to Potts-model with its own effective inverse temperature. 
This allows to reduce the structure of b.l. solutions 
for any $\Z$-model 
for periods $q=2,3$ to the known results for 
the Ising model and the Potts model. 

For $q\geq 4$ the gradient boundary law equations \eqref{eq:bl} no longer necessarily reduce to the Potts-type for any given transfer operator $Q$. However we can readily construct a family of transfer operators which are mapped to the Potts model under the fuzzy map $T_q$: In this case the fuzzy transfer operator $m \mapsto T_qQ (m)$ should be given by
\begin{equation}
T_qQ(0) = \frac{e^{\tilde\b}}{e^{\tilde\b} + q-1} \quad \text{and} \quad T_qQ(m) =  \frac{1}{e^{\tilde\b} + q-1}
\end{equation}
for any $m \in \{1,...,q-1\}$, where $\tilde \b>0$ is again the inverse temperature of the model. We can now define a transfer operator which meets these equations by simply putting
\begin{equation}\label{LiftQ}
   Q(m) =
   \begin{cases}
     \frac{e^{\tilde\b}}{e^{\tilde\b} + q-1} & \text{if } m=0, \\
     0 & \text{if } |m| >   \lfloor q/2 \rfloor, \\
      \frac{1}{e^{\tilde\b} + q-1} & \text{else }.
   \end{cases}
\end{equation}
Since this transfer operator is no longer strictly positive on $\Z$, it is necessary to adjust the state space $\Omega$ accordingly, such that the local specification remains well defined by \eqref{eq:specQ}. This can be guaranteed by setting 
\begin{equation}
\Omega = \{ \omega \in \Z^V \mid | \omega_i - \omega_j | \leq k \cdot d(i,j) \text{ for all } i,j \in V \}.
\end{equation}
As the proofs of Theorems \ref{PinGGM}, \ref{Mar} and \ref{MainThm} do not rely on the positivity of  the transfer operator, it is possible to construct GGMs as before, and furthermore, there is an effective inverse temperature such that multiple solutions to the b.l. equation exist. 
Not requiring strict positivity for all of $\Z$ makes our setup 
more general 
than the classical Gibbsian setup, since we do not have non-nullness 
for the specification. On the other hand, it is natural to incorporate also 
such cases, as they incorporate the standard nearest neighbor random walk. 
If we do insist on non-nullness, we can still define $Q$ in such a way (adding suitable 
exponentially 
decaying terms) that we recover our clock-models as inverse image of the Potts model:  
Define $Q(k) = e^{-\b k}$ for all $|k| > \lfloor q/2 \rfloor$ with $\b>0$, and 
\begin{equation}
\varphi_m(\beta) := \sum_{j \in \Z \setminus\{0\}} Q(qj+m)
\end{equation}
for $|m| \leq \lfloor q/2 \rfloor$. Clearly $\lim_{\b \to \infty} \varphi_m(\b) = 0$. 
Setting $Q(k) = T_qQ(k) - \varphi_k(\b)$ for all $|k| \leq \lfloor q/2 \rfloor$ and choosing $\b$ large enough, the constructed transfer operator $Q$ will be strictly positive and is mapped to the Potts model under the fuzzy transform. 

For the $q$-state Potts model on a binary tree the tree-automorphism invariant 
boundary laws are known and they form a rich class \cite{KuRo16}.  
Our above remark shows that there are gradient models which have $q$-periodic 
boundary laws corresponding to the Potts boundary laws. 
This means that there are very many non-trivial gradient measures.

\subsection{Correlation decay is governed by the fuzzy chain}
We remark that extremality of gradient Gibbs measures 
is equivalent to tail-triviality when we use the tail sigma-algebra
$\mathcal{T} :=\bigcap_{\L \SS V}\mathcal{T}_{\L}$. Note that this sigma-algebra 
keeps the asymptotic relative height information (and is bigger than the "naive" tail-algebra 
$\bigcap_{\L \SS V}\mathcal{F}_{\L^c}$. 
This can be seen by similar arguments as in \cite[Chapter 7.1]{Ge88} (using the sigma-algebra of events 
which are almost surely invariant under application of all kernels $\gamma'_{\L}$, 
and showing that this sigma-algebra is, up to $\mu'$-nullsets, equal to the tail-events 
$\mathcal{T}$).  

A warning is in order. 
One may be tempted to think that the expression of equation \eqref{drei} implies non-extremality of the l.h.s. in the set of GGMs, 
since it defines GGMs as certain mixtures. This argument would be wrong, as the mixture is made over measures which are not GGMs, and 
the phenomenon is more subtle and deserves further investigation. 
It is well known that the extremality of Gibbs measures is equivalent to having short-range correlations (see \cite[Proposition 7.9]{Ge88}). Even though we are not able to prove (non-)extremality of GGMs via this criterion, we show in the following how their correlation decay is governed by the fuzzy chain. This result might be useful for future analysis.

\begin{figure}
	\beginpgfgraphicnamed{cor}
		\begin{tikzpicture}[scale=1.15]
		
		\node[place] (x) at (-1,0) {$$};
		
		\node[place] (y) at (1,0)  {$$}
			edge  node[auto,swap] {$b_2$}		(x);
		
		\node[place] (w) at +([shift={(60:2)}]y) {$w$}
			edge  	node[auto,swap] {$b_3$}	(y);
		\node[place] (w1) at +([shift={(90:1)}]w) {$$}
			edge  		(w);
		\node[place] (w2) at +([shift={(30:1)}]w) {$$}
			edge  		(w);

		\node[place] (b) at +([shift={(300:2)}]y) {$$}
			edge  		(y);
		\node[place] (b1) at +([shift={(270:1)}]b) {$$}
			edge  		(b);
		\node[place] (b2) at +([shift={(330:1)}]b) {$$}
			edge  		(b);	
		
		\node[place] (a) at +([shift={(120:2)}]x) {$$}
			edge  		(x);
		\node[place] (a1) at +([shift={(90:1)}]a) {$$}
			edge  		(a);
		\node[place] (a2) at +([shift={(150:1)}]a) {$$}
			edge  		(a);	
			
		\node[place] (u) at +([shift={(240:2)}]x) {$u$}
			edge  	node[auto,swap] {$b_1$}	(x);
		\node[place] (u1) at +([shift={(210:1)}]u) {$$}
			edge  		(u);
			
		\node[place] (u2) at +([shift={(270:1)}]u) {$$}
			edge  		(u);	
		
		\draw[dashed] +([shift={(240:0.577)}]u) circle (1cm);
		\node at +([shift={(1,0)}]u2) {$\L$};
		
		\draw[dashed] +([shift={(60:0.577)}]w) circle (1cm);
		\node at +([shift={(1,0)}]w2) {$\D$};
		
		\end{tikzpicture}
	\endpgfgraphicnamed

	\caption{The disjoint subsets $U,W$ are connected via the path $\{ b_1,b_2,b_3\}$.}
	\label{fig:cor}
\end{figure}
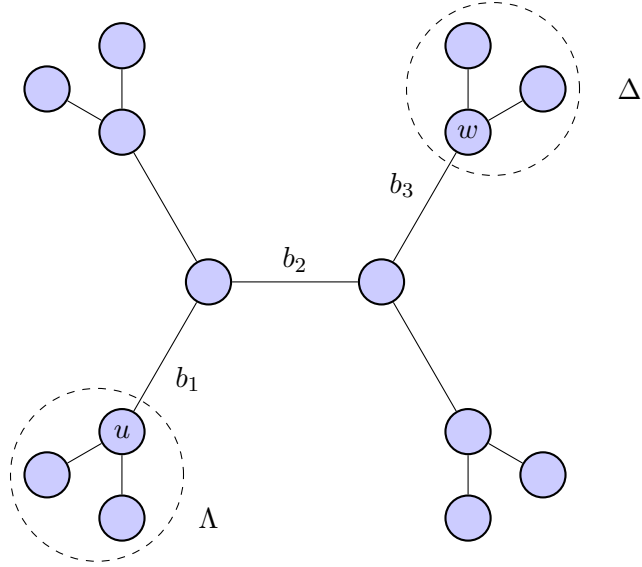

Let $\L \subset V$ and $\D \subset V$ be any finite connected sets with $d(\L,\D) = \inf \{ d(u,w) \mid u\in \L, w \in \D \}$. Assume $d(\L,\D)=n$. Hence there exists a unique path $\G= \{ b_1, b_2, ..., b_n\}$ of length $n$ connecting 
the two vertices $u\in \L$ and $w\in \D$ which realize the infimum. Furthermore let $\zeta_\L \in \Z^{\vec E(\L)}$ and $\zeta_\D \in \Z^{\vec E(\D)}$ be any two gradient configurations on the sub-volumes $\L$ and $\D$ respectively. We have
\begin{equation}\begin{split}\label{cov}
&\nu (\mathbf{1}_{\eta_\L = \zeta_\L} \; \mathbf{1}_{\eta_\D = \zeta_\D}) \\
&= \sum_{s\in\Z_q} \a(s) \nu_{u,s}(\mathbf{1}_{\eta_\L = \zeta_\L}) \sum_{\zeta_{b_1}, \zeta_{b_2},...,\zeta_{b_n}} \prod_{N=1}^n \bar P(T_q(s + \sum_{k=1}^{N-1} \zeta_{b_k}), \zeta_{b_N}) \nu_{w,\tilde s}(\mathbf{1}_{\eta_\D = \zeta_\D}),
\end{split}\end{equation}
where $\tilde s = T_q (s + \sum_{k=1}^n \zeta_{b_k})$. Using the transition matrix of the fuzzy chain $P'$ the r.h.s. of \eqref{cov} can be rewritten as
\begin{equation}
\sum_{s\in\Z_q} \a(s) \nu_{u,s}(\mathbf{1}_{\eta_\L = \zeta_\L}) \sum_{\tilde s \in \Z_q} \left(P'\right)^n(s, \tilde s) \nu_{w,\tilde s}(\mathbf{1}_{\eta_\D=\zeta_\D}).
\end{equation}
Hence, the covariance of $\mathbf{1}_{\eta_\L = \zeta_\L}$ and $\mathbf{1}_{\eta_\D = \zeta_\D}$ under the GGMs of the form \eqref{drei} is given by
\begin{equation}\begin{split}
 |\nu (\mathbf{1}_{\eta_\L=\zeta_\L} & \mathbf{1}_{\eta_\D=\zeta_\D}) - \nu(\mathbf{1}_{\eta_\L=\zeta_\L}) \nu ( \mathbf{1}_{\eta_\D=\zeta_\D})| \\
 &\leq  \sum_{s\in\Z_q} \a(s) \nu_{u,s}(\mathbf{1}_{\eta_\L=\zeta_\L}) \left| \nu_{u,s}(\mathbf{1}_{\eta_\D=\zeta_\D}) - \nu(\mathbf{1}_{\eta_\D=\zeta_\D}) \right| \\
& \leq \sum_{s\in\Z_q} \a(s) \nu_{u,s}(\mathbf{1}_{\eta_\L=\zeta_\L})  \left( \sum_{\tilde s \in \Z_q} \left| \left( P'\right)^n(s, \tilde s) - \a(\tilde s) \right| \right) \nu_{w, \tilde s} (\mathbf{1}_{\eta_\D=\zeta_\D}) \\
& \leq 2 \;  \nu(\mathbf{1}_{\eta_\L=\zeta_\L})  \left(\max_{s\in\Z_q} || \left( P'\right)^n(s, \cdot) -\a||_{TV}\right)  \left(\max_{\tilde s \in \Z_q} \nu_{w, \tilde s} (\mathbf{1}_{\eta_\D=\zeta_\D})\right),
\end{split}\end{equation}
where $|| \cdot ||_{TV}$ denotes the total variational distance. Note that the finite state transition matrix $P' \in M(q \times q; \R)$ always has strictly positive entries. This is certainly true for models where the transfer operator $Q$ is defined via an interaction potential \eqref{Qpotential}, and also holds for the not strictly positive transfer operators we constructed in \eqref{LiftQ}. 
Hence the fuzzy Markov chain associated with $P'$ is aperiodic and irreducible and it follows from the Convergence Theorem for Markov chains \cite[Theorem 4.9]{LPW} that there exist constants $C>0$ and $\delta \in (0,1)$ s.t. $\max_{s\in\Z_q}\Vert \left( P'\right)^n (s, \cdot) - \a \Vert_{TV} \leq C \delta^n$ for all $n \in \N$. 

\subsection{Ising classes. Identifiability of gradient Gibbs measures in terms of boundary laws.}
In the following we will restrict ourselves to the case of Ising classes, i.e. $q$-periodic boundary laws with $q=2$. So far we have seen that $q$-periodic boundary laws allow the construction of GGMs via equation \eqref{drei}. 
Note that there can be at most $3$ b.l. solutions, namely the trivial one $l\equiv 1$, and possibly non-trivial 
ones, of the form $(1,a)$ and $(a,1)$. These boundary laws favor even (respectively odd) layers. 
An inspection of formula \eqref{alt} shows that the gradient measure $\nu$ is the same for both non-trivial boundary law 
solutions. More generally, for any $q$ the boundary laws in the orbit of a $q$-periodic boundary law $l$, 
generated by the  shifts $j\in \Z$,  that is  
$(l(i))_{i\in \Z} \mapsto^j  (l(i+j))_{i\in \Z}$, all result in the same gradient Gibbs measure $\nu$.   

On the other hand, the non-trivial boundary law always results in a different gradient measure than the trivial one. 
To see this, 
let $l(i) = a$ for $i \text{ mod }2=1$, and $l(i) = 1$ for  $i \text{ mod }2=0$. 

Using the alternative representation \eqref{alt}, the single-bond marginal for this measure is given by 
\begin{equation}
\nu_l(\eta_b = \zeta_b) = \frac{1}{Z_l^b} Q(\zeta_b) \sum_{s\in\Z_q} l(s) l(s+ \zeta_b),
\end{equation} 
where $Z_l^b$ is a normalizing constant. If $\nu_l$ and $\nu_{\underline{1}}$ were the same measure we would have
\begin{equation}\begin{split}\label{single}
\frac{\nu_l(\eta_b = 0)}{\nu_{\underline{1}}(\eta_b = 0)} &= \frac{Z^b_{\underline{1}}}{Z^b_l} (1+a^2)/2 = 1, \\
\frac{\nu_l(\eta_b = 1)}{\nu_{\underline{1}}(\eta_b = 1)} &= \frac{Z^b_{\underline{1}}}{Z^b_l} a = 1.
\end{split}\end{equation}
which implies $a = 1$. 

A similar identifiability statement holds for general $q$. More precisely, 
a nontrivial boundary law given by 
$l(i) = a\neq 1$ for $i \text{ mod }q=0$, and $l(i) = 1$ else, yields a different GGM than 
the trivial boundary law. This is shown by a single-bond computation analogous to  \eqref{single}.

\section{Appendix}
\subsection{Coupling measure $\bar\nu$, Gibbsian preservation under transform of measure}

Answering a question of Aernout van Enter, 
let us make an additional comment on the structure which has unfolded. 
We have frequently used projections of the spins to different directions:
On the one hand an infinite-volume  spin-configuration $\omega\mapsto^{T_q} \omega' $ 
maps to a  mod-$q$ fuzzy spin $\omega'$, via our fuzzy map 
$T_q(i) = i\text{ mod } q$.
On the other hand, an infinite-volume 
spin-configuration 
$\omega=(\omega_w,\zeta) \mapsto^{\nabla} \zeta $ also maps to a gradient configuration. The additional 
information needed to recover the spin is provided by its value $\omega_w$ at a pinning site $w$. 

Let us call an infinite-volume gradient configuration $\zeta$ and an infinite-volume fuzzy spin configuration 
$\omega'$ {\em compatible} iff there exists an infinite-volume spin-configuration $\omega$ for which 
$T_q\omega= \omega' $
and $\nabla\omega=\zeta $. This is to say that the two configurations have a joint 
lift to a proper spin configuration. 

The defining function $Q(m)=e^{-U(m)}$ of the gradient model also has a natural 
{\em mod-$q$-fuzzy image}, namely
$$Q'(m) := T_q Q(m) = \sum_{j\in \Z}Q(qj + m)$$ 
Taking a logarithm $Q'$ 
describes a renormalized Hamiltonian for a clock model on $\Z_q$. 

Suppose now we are on a regular tree and have found an (in height-direction) 
$q$-periodic tree-automorphism 
invariant boundary law $l$. 
The definition of the GGMs $\nu$ which are mixed over the fuzzy chain (see formula \eqref{drei}) extends in a natural 
way to a {\em joint measure (or coupling measure)} $\bar\nu(d\omega',d\zeta )$. This 
coupling measure has the following properties:

\begin{enumerate}
\item $\bar \nu\Bigl( \omega' \text{and } \zeta\text{ are  compatible} \Bigr)=1$. 

\item The marginal on gradients $\bar\nu(d\zeta )$ is a tree-automorphism invariant GGM.  

\item The marginal on fuzzy spins $\bar\nu(d\omega')$ is a tree-automorphism invariant 
finite-state Gibbs measure for $Q'$. 
\end{enumerate}

The definition of $\bar \nu$ is given by spelling out its expectation $\bar \nu(F)$ on a bounded 
local observable $F(\zeta_{\L},\omega'_{\L})$ on gradient variables and layer variables 
in a finite volume $\L$. The formula says that we need to substitute the layer variables 
which are obtained by pinning the layer at one site, and the gradient information $\zeta$
and it reads 
\begin{equation}\begin{split}\label{dreiei}
&\bar \nu(F)
=\sum_{\zeta_{\L}}\sum_{\varphi'_w \in \Z_q}
\alpha^{(l)}(\varphi'_w)\prod_{\langle x,y\rangle \in \overrightarrow{E}_w: x,y \in \L}\bar P_{x,y}\Bigl(
T_q\bigl (
\varphi'_w+\sum_{b\in \Gamma(w,x)}\zeta_b
 \bigr),
\zeta_{\langle x,y \rangle}\Bigr)\cr
&\qquad F\Bigl(
\zeta_{\L}, 
\bigl(T_q\bigl (
\varphi'_w+\sum_{b\in \Gamma(w,x)}\zeta_b
 \bigr)\bigr)_{w\in \L}
\Bigr).
\end{split}
\end{equation}
Here we have assumed that the pinning site $w$ is in the finite volume $\L$. 

Then the first property follows by construction, the second is the content of Theorem \ref{MainThm},
and the last one follows using the relation between $P'$ and $\bar P$ given by \eqref{fuz}.

Let us comment now on similarities and differences between earlier uses of transformations 
of Gibbs measures. 
On the one hand, one may say that Property 3 feels like an example of a preservation of 
the Gibbs property under the map $T_q$, and one of the nice situations 
where "Gibbs goes to Gibbs". 
This is not completely true, as there is an important difference, and the situation is slightly more complicated. 
While fuzzy spins $\omega' \in \O' := (\Z_q)^V$ have a natural tree-invariant distribution, namely the fuzzy chain which is Gibbs for $Q'$, and 
gradient configurations $\zeta \in \O^\nabla$ have a tree-invariant GGM, this is {\em not true} for the spins. 
Spins $\omega \in \O$ do not have a natural tree-invariant measure, hence the fuzzy chain 
is not the direct image of a hypothetical measure on spins, as it usually is in studies of RG transforms 
when one starts from a well-defined measure on the spins. 
The best one can do to relate fuzzy spins and gradients is via the coupling measure $\bar \nu$. 

In this sense the theory  presented in this paper is a  {\em generalization} 
of a constructive use of  transformations for which {\em Gibbs goes to Gibbs}, 
via the coupling described above in \eqref{dreiei}.

\begin{figure}
\centering

\beginpgfgraphicnamed{dia}
\begin{tikzpicture}[scale=1.05]
													\node 		(fuzzyMC)		at (0,0)	{$\GG'_{\text{MC}}(T_q Q)$};	 \node   (s1) at (1.5,0) {$\subset$};	\node 		(tiPM)		at (3,0)	{$\M_1^{}(\O') $};
\node 		(BLqti)		at (-4,-3)	{$\bl^{\text{q}}$};		\node 		(nuBLqti)		at (0,-3)	{$\bar \nu^\cdot (\bl^{q}) $};	 \node   (s2) at (1.3,-3) {$\subset$};	\node 		(M1ticp)		at (3,-3)	{$ \M_1^{ \cp}(\O',\O^\nabla) $};
													\node 		(GnQ)		at (0,-6)	{$\GG^\nabla(Q) $};			 \node   (s3) at (1.5,-6) {$\subset$};	\node 		(M1ti)		at (3,-6)	{$\M_1^{}(\O^\nabla) $};
\node 		(BLnti)		at (-4,-9)	{$\bl^{\text{Norm}} $};	\node 		(GMCQ)		at (0,-9)	{$\GG_{\text{MC}}(Q) $};		 \node   (s4) at (1.5,-9) {$\subset$};	\node 		(M1tiQ)		at (3,-9)	{$\M_1^{}(\O) $};

\draw[->] (nuBLqti) to (fuzzyMC);
\draw[->] (BLqti)     to (nuBLqti);
\draw[->] (nuBLqti) to (GnQ);

\draw[->] (M1ticp) to (tiPM);
\draw[->] (M1ticp) to (M1ti);
\draw[->] (BLqti) to [out=270, in=180] (GnQ);

\draw[<->] (BLnti) to (GMCQ);

\node (thm3) at (-3.7,-5.7)    {$\nu^\cdot : l \mapsto \nu^l$};

\node (sr1) at (0,-7.5)		{$\rotsubset$};
\node (sr2) at (3,-7.5)		{$\rotsubset$};

\draw[dashed, gray, thin] 	(-5,-7.5) -- (7.5,-7.5);

\draw (5,1) -- (5,-10); 

\node (fuzzyspins) 	at (6.5,0)    {fuzzy spins $\o'$};
\node (gradients) 	at (6.7,-3)    {compatible $(\o',\eta)$};
\node (gradients) 	at (6.5,-6)    {gradients $\eta$};
\node (spins) 		at (6.5,-9)    {spins $\o$};

\node (t1) at (-1.1,-1.5)  {$(\o', \eta) \mapsto \o'$};
\node (t2) at (-1.1,-4.5)  {$(\o', \eta) \mapsto \eta$};

\node (l1) at (-2,-2.6)   {$\bar \nu^\cdot : l \mapsto \bar \nu^l$};
\node (l2) at (-2,-8.6)   {$\nu^\cdot_{\text{Zac}}: l \mapsto \nu^l$};

\end{tikzpicture}
\endpgfgraphicnamed

\caption{The relationship between $q$-periodic b.l.'s and the (gradient) Gibbs measures is displayed above the dashed line. The classical theory by Zachary for normalizable b.l.'s is visualized below the dashed line. Here all objects appearing are assumed to be tree automorphism invariant.}
	\label{fig:dia}
\end{figure}
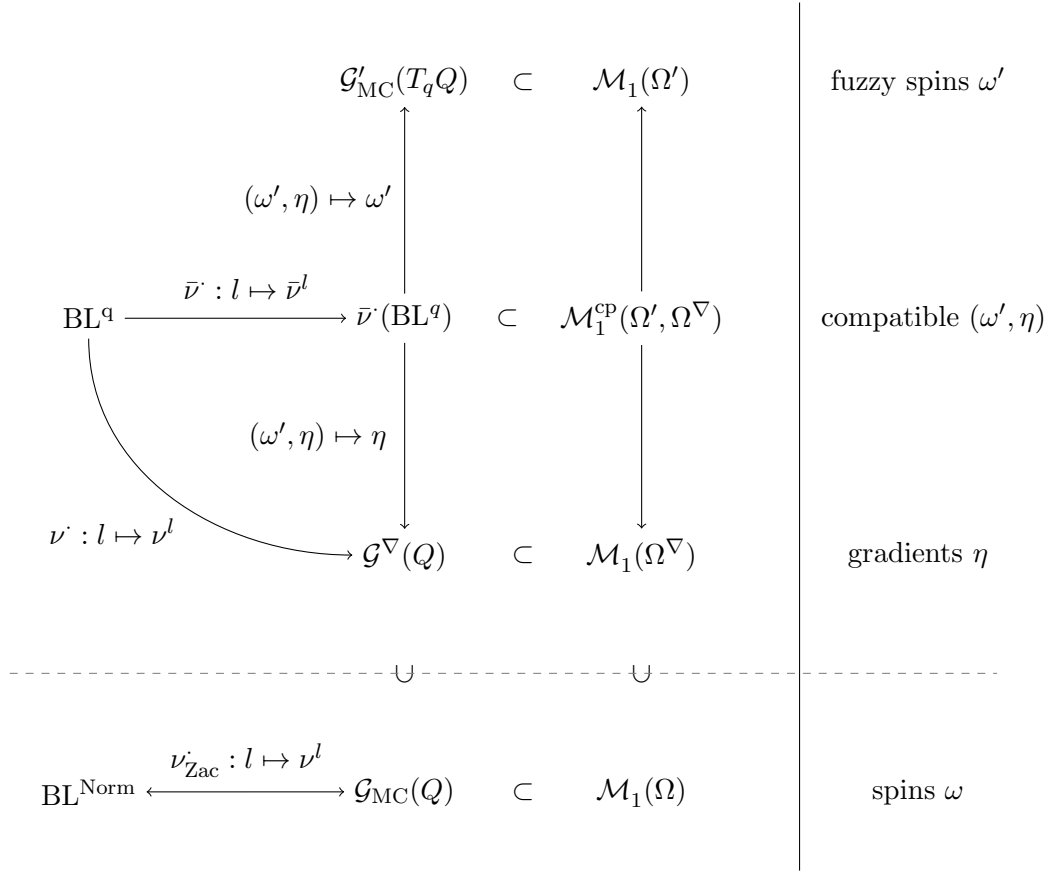

In Figure \ref{fig:dia} the main result of our paper, Theorem \ref{MainThm},
is visualized as the curved arrow. 
Here we have denoted the set of $q$-periodic tree invariant b.l.'s by $\bl^{\text{q}}$, the set of 
Gibbs measures on the fuzzy spins which are tree-indexed Markov chains (splitting Gibbs measures) 
by $\GG'_{\text{MC}}(T_q Q)$, the set of tree-invariant measures on the fuzzy spins by $\M_1^{}(\O')$, the set of coupling measures on $(\O^\nabla, \O')$ which correspond to a b.l. via \eqref{dreiei}
by $\bar \nu^\cdot (\bl^{q})$, the set of measures on the set of {\em compatible} gradient and fuzzy configurations which are tree-invariant by $\M_1^{\cp}(\O^\nabla, \O')$ and the set of tree-invariant measures on $\O^\nabla$ by $\M_1^{}(\O^\nabla)$. 

Below the dashed line we have also given a visualization of the classical theory of Zachary \cite{Z83} 
and its correspondence to our results. 
Every normalizable tree-invariant boundary law $l \in \bl^{\text{Norm}}$ corresponds to a Gibbs measure which is a Markov chain. This set of measures is denoted by $\GG_{\text{MC}}(Q)$. Conversely, every
$\nu \in \GG(Q)$ which is also a Markov chain can be represented by a b.l. which is unique (up to a positive pre-factor). The set of measures $\GG_\text{MC}(Q)$ can be thought of as a subset of the {\em gradient} Gibbs measures $\GG^\nabla(Q)$, as any Gibbs measure gives rise to a gradient measure, but not vice versa. 
Note that by the theory of Zachary applied to finite local state space, there is also a one-to-one correspondence between the elements of $\GG'_{\text{MC}}(T_qQ)$ and $\bl^q$. 
All the objects we construct above the dashed line are new. They are not contained in the theory of Zachary, 
as $\nu^\cdot (\bl^{q}) \subset\GG^\nabla(Q) \setminus \GG_\text{MC}(Q)$, that is our gradient Gibbs measures 
live in the delocalized regime and can not be understood as projection of Gibbs measures to the gradient variables.

\subsection {Mixtures of layer-dependent chains on the tree $\Z$ lack the Gibbs property}

It is interesting to discuss the necessity of our trees having degree $d+1\geq 3$ for 
our construction by means of the following example on the integers. 
Answering a question of Roberto Fern\'andez it shows 
that we may construct translation-invariant gradient measures in one dimension as non-trivial 
mixtures of pinned measures appearing from layer-dependent 
transition probabilities, but they will be lacking the Gibbs property. 

We will consider two classes, i.e. $q=2$. We build a 
translation-invariant gradient measure in terms of mixtures as follows. 
This illustrates the first aspect of Theorem \ref{MainThm}.  
Let us use notation as in Theorem 3. 
Take two layer-dependent Markov chains living on $\Z$ whose transition probabilities 
to make a step $\zeta\in \{-1,0,1\}$ starting from a layer of type $1$ (or $0$ respectively) 
are given by 
\begin{equation}\begin{split}
&\bar P(1,\zeta)=\e_1 1_{|\zeta|=1}+ (1- 2\e_1)1_{|\zeta|=0},\cr
&\bar P(0,\zeta)=\e_0 1_{|\zeta|=1}+ (1- 2\e_0)1_{|\zeta|=0}\cr
\end{split}\end{equation}
with $\e_1\neq \e_0$. 
We define a corresponding fuzzy chain on the state space $\{0,1\}$ by  
$P'(0,0)=1-2 \e_0$, $P'(0,1)=2 \e_0$, $P'(1,0)=2 \e_1$, $P'(1,1)=1-2 \e_1$. 
Its invariant distribution is given by $\frac{\a(1)}{\a(0)}
=\frac{\e_0}{\e_1}$. 
It is then easy to see that formula \eqref{dreiei}  defines a translation-invariant measure on fuzzy spins and gradients. 
Is the marginal of this  measure on the gradients a GGM for some gradient Hamiltonian? 
In general mixing measures over external parameters tends to destroy quasilocality, which has been observed 
in various non-trivial scenarios, for example joint measures of random systems \cite{Ku99}. 


The answer is no, also here, for any choice of $\e_1\neq \e_0$, 
as the following computation of conditional probabilities shows. 
We take $W=L\cup\{b\}\cup R$ to be a finite connected set of bonds, with reference 
bond $b$ in the middle and consider the fraction of single-bond conditional probabilities  
\begin{equation}\begin{split}
&\frac{\nu(\eta_b=0|\eta_L=0,\eta_R=0)}{\nu(\eta_b=1|\eta_L=0,\eta_R=0)}
=\frac{\a(1)(1-2\e_1)C^{|R|+|L|}+\a(0)(1-2\e_0)}{\a(1)\e_1 C^{|L|}+\a(0) \e_0 C^{|R|}}
\end{split}\end{equation}
with the constant $C=\frac{1-2\e_1}{1-2\e_0}$ which we can assume to be strictly bigger than one 
without loss of generality. But this expression tends to infinity when we send 
both $|L|$ and $|R|$  to infinity, implying that $\nu(\eta_b=0|\eta_L=0,\eta_R=0)\rightarrow 1$, 
which would be impossible for a gradient Gibbs measure.




%

%




\section*{Acknowledgement}
This work is supported by Deutsche Forschungsgemeinschaft, RTG 2131
{\em High-dimensional Phenomena in Probability - Fluctuations and Discontinuity}. CK thanks the participants 
of the workshop {\em Transformations in Statistical Mechanics: Pathologies and Remedies}, on the occasion of the 
65'th birthdays of Aernout van Enter and Roberto Fern\'andez, 
held at the Lorentz Center Leiden, 10-14 October 2016, for useful discussions.

\end{document}